\documentclass{birkjour}
\usepackage[all]{xy}
\usepackage[utf8]{inputenc}
\usepackage{hyperref}
\usepackage{amsmath,amssymb}
\usepackage{enumerate,multicol}
\usepackage{xcolor}
\usepackage{lineno}

\newtheorem{theorem}{Theorem}[section]
\newtheorem{corollary}[theorem]{Corollary}
\newtheorem{lemma}[theorem]{Lemma}
\newtheorem{conjecture}[theorem]{Conjecture}
\newtheorem{proposition}[theorem]{Proposition}
\newtheorem{definition}[theorem]{Definition}
\newtheorem*{remark}{Remark}

\newtheorem{problem}{Problem}

\newtheorem{ithm}{Theorem}[section]

\newtheorem{icor}[ithm]{Corollary}

\DeclareMathOperator{\Ricci}{Ric}

\def\rm#1{\mathrm{#1}}
\def\cal#1{\mathcal{#1}}
\def\bb#1{\mathbb{#1}}
\def\lie#1{\mathfrak{#1}}

\usepackage[all]{xy}

\newcommand{\h}{\frac{1}{2}}

\DeclareMathOperator{\id}{id}

\DeclareMathOperator{\Hess}{Hess}
\DeclareMathOperator{\tr}{tr}
\DeclareMathOperator{\scal}{scal}
\title{Positive Ricci curvature through Cheeger deformations}
\author{Leonardo F. Cavenaghi}
\address{Instituto de Matemática, Estatística e Computação Cinetífica -- Unicamp, Rua Sérgio Buarque de Holanda, 651, 13083-859, Campinas, SP, Brazil}
\email{leonardofcavenaghi@gmail.com}
\author{Renato J.M. e Silva}
\address{Instituto de Matem\'atica e Estat\'istica e Computaç\~ao -- UNICAMP, Rua S\'ergio Buarque de Holanda, 651,13083-97 Campinas, SP, Brazil}
\email{renatojuniorms@gmail.com}
\author{Llohann D. Sperança}
\address{Instituto de Ci\^encia e Tecnologia -- Unifesp, Avenida Cesare Mansueto Giulio Lattes, 1201, 12247-014, S\~ao Jos\'e dos Campos, SP, Brazil}
\email{lsperanca@gmail.com}

\begin{document}
  \begin{abstract}
    This paper is devoted to a deep analysis of the process known as Cheeger deformation, applied to manifolds with isometric group actions. Here, we provide new curvature estimates near singular orbits and present several applications. 
    As the main result, we answer a question raised by a seminal result of Searle--Wilhelm about lifting positive Ricci curvature from the quotient of an isometric action. To answer this question, we develop techniques that can be used to provide a substantially streamlined version of a classical result of Lawson and Yau, generalize a curvature condition of Chavéz, Derdzinski, and Rigas, as well as, give an alternative proof of a
result of Grove and Ziller.
  \end{abstract}

  \maketitle
   
  \section{Introduction}
  An interesting and challenging problem in Riemannian geometry is that of producing examples of metrics of positive (Ricci or sectional) curvatures. The current scarcity of known examples of metrics with positive sectional curvature compared to the known examples of metrics with non-negative sectional curvature illustrates the difficulty of the problem for the sectional curvature.
   
  Of particular interest is the search for positively curved metrics on \textit{exotic manifolds}. Gromoll and Meyer \cite{gromoll1974exotic} constructed the first exotic sphere with a metric of non-negative sectional curvature; Wilhelm \cite{wilhelm-lots} constructed metrics of positive Ricci curvature and almost non-negative sectional curvature in every exotic sphere which is a 3-sphere bundle over the 4-sphere; Grove and Ziller \cite{gz} produced metrics of nonnegative sectional curvature on these examples and Goette, Kerin and Shankar \cite{shankarannals} extended Grove--Ziller's result to all 7-spheres. Apart from spheres, Grove, Verdiani and Ziller \cite{grove2011exotic}, and independently Dearricott \cite{dearricott20117}, built an exotic unit tangent space with positive sectional curvature. 
   
  Concerning non-trivial examples of manifolds with metrics of positive Ricci curvature, Nash \cite{nash1979positive}, Poor \cite{poor1975some}, Searle and Wilhelm \cite{searle2015lift}, Wraith \cite{wraith1997,wraith2007new}, Joachim \cite{jowr} and Crowley and Nordstr\"om and Crowley and Wraith \cite{crowley2015new, crowley2017positive} proved the existence of such metrics on certain bundles and exotic manifolds. In \cite{SperancaCavenaghiPublished, cavenaghi2019positive} the first and third named authors built metrics of positive Ricci curvature on several exotic manifolds and on the total space of bundles where the fibers and/or the base spaces are exotic manifolds, or even when the base manifold is a so-called Shrinking Ricci Soliton (\cite{SperancaCavenaghiPublished, cavenaghi2019positive} are particularly related to the ones of Searle and Wilhelm, Nash and Poor). With different aims, the works of Gilkey, Park and Tuschm\"ann \cite{Gilkey1998} and Belegradek and Wei \cite{wei} are also interesting references for positive Ricci curvature on bundles. 
   
  Riemannian submersions play an important role in the construction of manifolds with positive curvature properties due to the O'Neill submersion formula, which implies that Riemannian submersions do not decrease sectional curvature. However, it is not necessarily true that Riemannian submersions preserve positive Ricci curvature (see \cite{pro2014riemannian}).
   
  On the other hand, one may ask whether one can lift positive curvature from the base of a Riemannian submersion to its total space. There is an easy counterexample to this question: simply consider the projection onto the first coordinate $\mathbb{R}P^2\times\mathbb{R}P^2\to \bb RP^2$. The base of this submersion has positive sectional curvature, however \textit{Synge's Theorem} implies that the positive sectional curvature cannot be lifted to the total space since $\pi_1(\mathbb{R}P^2\times\mathbb{R}P^2)\cong\mathbb{Z}_2\times \mathbb{Z}_2$.
   
  A next natural question 
  is \textit{can one lift positive Ricci curvature from the orbit space?} Searle and Wilhelm \cite{searle2015lift} answered this question in affirmative for a large class of submersions.
   
  \begin{theorem}[Searle--Wilhelm, \cite{searle2015lift}]\label{ithm:searleintro}
    Let $(M,g)$ be a compact Riemannian manifold endowed with a $G$-action satisfying
    \begin{enumerate}[$(SW1)$]
      \item $G$ is a compact connected Lie group acting effectively and by isometries;
      \item A $G$-principal orbit has finite fundamental group;
      \item $\Ricci_{M/G}\geq 1$ in the orbital distance metric.
    \end{enumerate}
    Then $M$ carries a $G$-invariant metric of positive Ricci curvature.
  \end{theorem}

  The main idea in the proof of Theorem \ref{ithm:searleintro} is to perform a conformal change on the metric $g$ followed by a standard deformation, commonly used for constructing metrics of nonnegative/positive sectional curvature, called \textit{Cheeger deformation} (see \cite{cheeger,Muter} for instance). 
  Based on the proof's delicate estimates and the lack of examples, the question was raised: \textit{is it possible to prove Theorem \ref{ithm:searleintro} using only Cheeger deformations?}

  Motivated by the two questions above, we provide a deep analysis of the behavior of the Cheeger deformation near singular orbits. As a consequence, we answer in negative the second question (Theorem \ref{ithm:mainitro}); simplify the proof of a result in \cite{Grove2002} (Theorem \ref{ithm:GZ}); recover the celebrated result on the existence of metrics of positive scalar curvature under non-Abelian symmetry assumptions \cite{lawson-yau} (Theorem \ref{ithm:LW}); and improve the condition for positive sectional curvature in \cite{CDR} (Theorem \ref{ithm:CDR}).
   
  As the first result, we show that it is not always possible to lift positive Ricci curvature from the quotient only by using Cheeger deformation. Or, equivalently, we show that the conformal change in the proof of Theorem \ref{ithm:searleintro} was applied in an essential way
  (see the family of examples presented in Section \ref{sec:examples} for details):
  \begin{ithm}\label{ithm:mainitro}
    There are Riemannian manifolds satisfying the hypotheses of Theorem \ref{ithm:searleintro} that do not develop positive Ricci curvature after any Cheeger deformation.
  \end{ithm}

  Theorem \ref{ithm:mainitro} is proved through a fully algebraic characterization of sufficient conditions for a manifold to admit a $G$-invariant metric that does not develop positive Ricci curvature after Cheeger deformation. Such a characterization is obtained by reducing the study to tangent vectors at singular orbits, that are fixed by the isotropy representation, which we call \textit{fixed axes}. Afterwards, we recognize these algebraic conditions in terms of geometric obstructions. This is the content of Theorem \ref{thm:technical}, which has Proposition \ref{prop:trace.to.Ricci} as a local converse.
   
Related to what we just mentioned, given $p\in M$, denote by $\rho \colon G_p\to O(\cal H_p)$ the restriction of the isotropy representation at $p$ to $\cal H_p=(T_pGp)^\perp$, and let $G^0_p$ denote the identity component of the isotropy subgroup $G_p$. We prove:
  \begin{ithm}\label{ithm:obstruction}
    Let $(M,g)$ and $G$ satisfy $(SW1)$-$(SW3)$.
    Then, if $g$ has directions with negative Ricci curvature after any finite Cheeger deformation, it follows that
    \begin{enumerate}[$1.$]
      \item there is a point $q$ in a singular orbit and a non-zero vector $X\in\cal H_q$ which is fixed by $\rho({G_q^0});$
      \item the restriction of $\rho$ to $\cal H_q\cap \mathrm{span}\{X\}^{\perp}$ is reducible.
    \end{enumerate}
  \end{ithm}

  In particular, we conclude some simple criteria for Cheeger deformations to allow the development of positive Ricci curvature under the hypotheses $(SW1)$-$(SW3)$:
   
  \begin{icor}\label{icor:sufficientconditions} Let $(M,g)$ and $G$ satisfying $(SW1)$-$(SW3)$. Then $g$ develops positive Ricci curvature after a finite Cheeger deformation if any of the following conditions hold:
    \begin{enumerate}[$(a)$]
      \item the singular strata are composed of isolated orbits;
      \item isotropy representations have no non-zero fixed points;
      \item the isotropy representation $\rho : G^0_q \to O(\cal H_q)$ is irreducible at every point in the singular strata.
    \end{enumerate} 
  \end{icor}
   
  Moreover, item $(a)$ can be viewed as a condition on the quotient $M/G$:
   
  \begin{icor}\label{icor:quotient}
    Suppose that $(M,g)$ and $G$ satisfy $(SW1)$-$(SW3)$. Then $g$ develops positive Ricci curvature after a finite Cheeger deformation if the singular strata of the quotient $M/G$ is 0-dimensional.
  \end{icor}

The analysis here employed is rich enough to furnish a better understanding of Cheeger deformations, making it possible to produce other positive curvatures, or at least furnishing some criteria for it. For instance, we recall that in \cite{lawson-yau}, Lawson and Yau construct a metric with positive scalar curvature on any compact manifold $M$ endowed with an action of a compact non-Abelian Lie group $G$. The proof considers a copy of $SU(2)\subseteq G$ (or $SO(3)\subseteq G$) and is broken in two parts: one first applies the \textit{canonical deformation} to the $S^3$ fibers (see \cite{besse1987einstein} for details) on the regular part; then delicate estimates are made near the singular strata, independent of the deformation parameter. A downside of the construction is the loss of symmetry. 
  By using Cheeger deformations, we maintain the original symmetry and reduce the problem near singular orbits to elementary estimates. More precisely, we re-prove:
   
  \begin{theorem}[Lawson--Yau]\label{ithm:LW}
    Let $(M,g)$ be a compact Riemannian manifold. Suppose that $G$ is a compact Lie group with non-Abelian Lie algebra that acts effectively on $M$ by isometries. Then $g$ develops positive scalar curvature after a finite Cheeger deformation.
  \end{theorem}
   
  In the realm of sectional curvature, we generalize the condition for positive sectional curvature in \cite{CDR}, by replacing the canonical deformation, which would only work for principal $S^3$- or $SO(3)$-bundles with totally geodesic fibers, with Cheeger deformations. An interesting property of the condition below is an almost complete decoupling of the three main ingredients that characterize the geometry of the regular part $M^{reg}$: the geometry of the fiber; the dynamics of the horizontal distribution; and the geometry of the base. This fact should be extremely useful in applying analysis methods to problems of existence of positive sectional curvature on principal bundles.
  
  To better understand the statement of Theorem \ref{ithm:CDR}, we recall that given a biinvariant inner product $Q$ on $\lie g$, for every point $p$ in a principal orbit, we denote by $\Omega:\cal H_p\times\cal H_p\to \lie m_p$ the \textit{curvature 2-form} of the bundle $M^{reg}\to M^{reg}/G$, where $M^{reg}$ stands for the open, dense and convex set where each two orbits are diffeomorphic to each other. More precisely, given two horizontal vector fields $X,Y$, we define $\Omega(X,Y)$ as the unique element in $\lie m_p$, where $\lie m_p$ is identified via isometric action vectors with the tangent space to the orbit through $p$, such that \[\Omega(X,Y)^*=-[X,Y]^{\mathcal V} = -2A_XY,\]
  where $A : \cal H_p\times \cal H_p \rightarrow \cal V_p$ stands to the so-called O'Neill tensor and the superscript $\ast$ refers to action vectors. Observe that it can also be equivalently characterized by:
  \[g(\Omega_X^*V^*,Y)=Q(\Omega(X,Y),V)=-2g(A^*_XP^{-1}V^*,Y),\]
  where $P$ is uniquely characterized by $g|_{\mathcal V_p}(\cdot, \cdot) = Q(P\cdot, \cdot)$. In particular, $-2A^*P^{-1}V^*=\Omega^*V$. 
   
  \begin{ithm}\label{ithm:CDR}
    Suppose that $G$ is compact and that $G/G_p$ has positive sectional curvature as a normal homogeneous manifold for every $p\in M^{reg}$. Then $\sec_{g_t}>0$ for any sufficiently large $t$ if, and only if, there is $k>0$ such that
    \begin{multline}\label{eq:CDR}
      (R_{M/G}(X,Y,Y,X)-k\|X\wedge Y\|^2_g)\\
      (\tfrac{1}{2}Q(\Hess P^{-1}(X)V,V) +\tfrac{1}{4}\|\Omega^*_XV\|_{g}^2-k\|X\|^2_gQ(P^{-1}V,V))\\\geq Q((\nabla_X\Omega)_XY,V)^2
    \end{multline}
    for every $X,Y\in \cal H_p$, $V\in\lie g$ and $p\in M^{reg}$.
  \end{ithm}
   
  By recalling that Cheeger deformation preserves positive sectional curvature, one concludes that the condition in Theorem \ref{ithm:CDR} is a necessary condition for a manifold to have positive sectional curvature (even before applying a Cheeger deformation).
   
  Finally, Theorem \ref{ithm:searleintro} demands positive Ricci curvature on the quotient. This is not possible if the quotient is an interval, as in the case of cohomogeneity-one actions. By studying the limit behavior of the Cheeger deformation, we also use our techniques in this direction, recovering a special case of a result of Grove and Ziller:
   
  \begin{theorem}[Grove--Ziller, \cite{Grove2002}]
    \label{ithm:GZ}
    Suppose that $M$ is a compact cohomogeneity one manifold such that
    \begin{enumerate}[(i)]
      \item $M$ has two singular orbits; 
      \item a principal orbit has a finite fundamental group.
    \end{enumerate}
    Then, there is an invariant metric $g$ on $M$ with positive Ricci curvature.
  \end{theorem}
   
  The paper is structured as follows: the definition, construction and basic facts about Cheeger deformations are gathered in Section \ref{sec:cdef}. The main estimates around singular orbits are in Section \ref{sec:estimates}. Section \ref{sec:itema} applies the theory so far to prove Theorems \ref{ithm:CDR}, \ref{ithm:GZ}, \ref{ithm:LW} and item $1$ of Theorem \ref{ithm:obstruction}. In Section \ref{sec:proofmainintro} we begin providing some algebraic description of the Ricci tensor on fixed axes, proving item $2$ of Theorem \ref{ithm:obstruction}. The examples that ensure Theorem \ref{ithm:mainitro} are described in Section \ref{sec:examples}. Finally, Section \ref{sec:proofmainintro} fully describes sufficient conditions, despite presenting some obstructions, to positive Ricci curvature be lifted from orbit spaces. Later on, the characterizations they provided are recognized as geometric data.

  \section{Cheeger deformations on $G$-manifolds}
  \label{sec:cdef}
  We follow \cite{Muter} and \cite{mutterz} to give a brief review on the procedure known as \textit{Cheeger deformations}. Specifically, we recall some results that we shall need in the rest of this work.  
   
  \subsection{Cheeger deformations and its associated tensors}
   
  Let $(M,g)$ be a Riemannian manifold endowed with an isometric action by a compact Lie group $G$ with a biinvariant Riemannian metric $Q$. We recall that, for each point $p \in M$, the map $g \mapsto gp$ induces a diffeomorphism of $G\big/G_p$ onto the orbit $Gp$, where $G_p$ is the isotropy subgroup at $p$. Moreover, the metric $Q$ induces an orthogonal decomposition $\mathfrak{g} = \mathfrak{g}_p\oplus\mathfrak{m}_p$, where $\mathfrak{g}_p$ stands for the Lie algebra of $G_p$ and $\mathfrak{m}_p$ is isomorphic to $T_pGp$ via action fields: 
  \begin{equation*}
    \lie g\ni U \mapsto U^*_p=\frac{d}{dt}\Big|_{t=0}e^{tU}p,
  \end{equation*}
  where $e^{tU}$ denotes the Lie group exponential map.
   
  We call the tangent space to the orbit at $p$ \emph{the vertical space at $p$} and denote $T_p Gp=\cal V_p$; the $g$-orthogonal complement of $\mathcal{V}_p$ in $T_pM$ is called the \emph{horizontal space at} $p$ and is denoted by $\mathcal{H}_p$. Thus, every tangent vector $\overline X \in T_pM$ can be uniquely written as $\overline X = X + U^{\ast}_p$, where $X\in \cal H_p$ and $U\in \lie m_p$. We omit the subscript $p$ in $U_p^*$ whenever there is no risk of ambiguity.
   
   The Cheeger deformation consists of a $1$-parameter family of $G$-invariant Riemannian metrics on $M$ produced by an appropriate shrinking of the metric $g$ in the orbit's direction. It preserves the horizontal space at each point together with its metric. This procedure promptly generalizes the classical \textit{Canonical Variation} (see \cite[Example 2.1.1, p. 56]{gw}) on $G$-principal bundles. 
   
To construct the family of metrics, consider the following free isometric $G$-action on $(M\times G,g+\tfrac{1}{t}Q)$:
  \begin{equation}\label{eq:actioncheeger}
    r(p,g) := (r p, rg),\quad\forall r\in G.
  \end{equation}
  The map $\bar\pi \colon (p,g) \mapsto g^{-1}p$ defines a diffeomorphism between the quotient and $M$, thus inducing a $1$-parameter family of Riemannian metrics $g_t$. One observes that
  $g_t(\cal H_p,\cal V_p)=0$ and that $g_t|_{\cal H_p} = g|_{\cal H_p}$ for every $p\in M.$
   
  \begin{definition}
    We call the resulting metric $g_t$ the \emph{Cheeger deformation} of $g$ (at time $t$).
  \end{definition}
   
  The metric $g_t$ can be completely described by the following tensors:
  \begin{definition}~ 
    \begin{itemize}
      \item The \emph{orbit tensor} at $p$ is the linear map $P : \mathfrak{m}_p \to \mathfrak{m}_p$ defined by
      \[g(U^{\ast},V^{\ast}) = Q(PU,V),\quad\forall U^{\ast}, V^{\ast} \in \mathcal{V}_p.\]
      %    that can be extended to $\mathfrak{g}$ by declaring $P(\mathfrak{g}_p) = 0$
      \item The \emph{deformed orbit tensor} of $g_t$ at $p$ is defined to be the linear map $P_t:\lie m_p\to \lie m_p$ such that
      \[g_t(U^{\ast},V^{\ast}) = Q(P_tU,V), \quad\forall U^{\ast}, V^{\ast} \in \mathcal{V}_p.\]
      \item The \emph{metric tensor} of $g_t$ at $p$ is the linear map $C_t: T_pM\to T_pM$ satisfying
      \[g_t(\overline{X},\overline{Y}) = g(C_t\overline{X},\overline{Y}), \quad\forall \overline{X}, \overline{Y} \in T_pM.\]
    \end{itemize}
  \end{definition}
  All three tensors are symmetric, positive definite, and related in the following way:
   
  \begin{proposition}[Proposition 1.1, \cite{mutterz}] \label{propauxiliar}~ 
    \begin{enumerate}
      \item $P_t = (P^{-1} + t1)^{-1} = P(1 + tP)^{-1}$,
      \item Given $\overline{X} = X + U^{\ast}$ then $C_t(\overline{X}) = X + ((1 + tP)^{-1}U)^{\ast}$.
    \end{enumerate}
  \end{proposition}
  
  \begin{remark}
    Concerning our comment on the fact that Cheeger deformations can be seen as a $1$-parameter subgroup, just observe that: at least in the case of principal bundles, we can see the orbit tensor $P$ as a section of the bundle over $M$ in which the fibers are automorphisms of the Lie algebra $\lie g$ of $G$. In this manner, using the fact that $P$ is a symmetric tensor, one can check that the Cheeger deformed tensor $P_t$ corresponds to the flow associated to the following ODE:
    \begin{equation}
        \begin{cases}
            \frac{d}{dt}\psi(t) = -(\psi(t))^2,\\
            \psi(0) = P
        \end{cases}
    \end{equation}
  \end{remark}
   
  The metric tensor $C_t$ also plays a crucial role in the computation of the sectional curvature of $g_t.$ As initially observed by Cheeger and essential in the work of M\"uter (see \cite{Muter}), the expression of the sectional curvature of $g_t$ is much more natural when computed in the reparameterized plane $C_t^{-1}\overline{X}\wedge C_t^{-1}\overline{Y}$, instead of the original $\overline{X}\wedge \overline{Y}$. Specifically, it is better to consider the following quantity:
  \begin{equation}
    \kappa_t(\overline X,\overline Y) := R_{g_t}(C_t^{-1}\overline{X},C_t^{-1}\overline{Y},C_t^{-1}\overline{Y},C_t^{-1}\overline{X})
  \end{equation}
  where $R_{g_t}$ stands for the (4,1) Riemannian curvature tensor 
  \[R_{g_t}(X,Y,Z,W)=g_t(\nabla_X\nabla_YZ-\nabla_Y\nabla_XZ-\nabla_{[X,Y]}Z,W).\]
  In particular, one concludes that the reparameterized sectional curvature is non-decreasing in $t$:
  \begin{theorem}[Proposition 1.3, \cite{mutterz}]\label{thm:curvaturasec}
    Let $\overline{X} = X + U^{\ast},~ \overline{Y} = Y + V^{\ast}$ be tangent vectors. Then,
    \begin{equation}\label{eq:curvaturaseccional}
      \kappa_t(\overline X,\overline{Y}) = R_g(\overline{X},\overline{Y},\overline{Y},\overline{X}) +\frac{t^3}{4}\|[PU,PV]\|_Q^2 + z_t(\overline{X},\overline{Y}),
    \end{equation}
    where $z_t$ is bilinear in each entry, non-decreasing, and zero at $t=0$. Moreover, at points on the regular stratum, it can be written as
    \[ z_t(\overline X,\overline Y) = 3t\left\|(1+tP)^{-1/2}P\nabla^{\mathbf{v}}_{\overline X}\overline Y - (1+tP)^{-1/2}t\h[PU,PV]\right\|_Q^2.\]
  \end{theorem}
  \begin{remark}
    A precise definition of $z_t$, together with important properties, is given in Lemma \ref{lem:zt}. 
  \end{remark}
   
  \subsection{The Ricci and scalar curvatures of a Cheeger deformation}\label{sec:2.1}
  Finally, we take advantage of Theorem \ref{thm:curvaturasec} to present formulae for the limit Ricci curvature of $g_t$ as $t\to \infty$ as well as for the scalar curvature $\mathrm{scal}_{g_t}$ for any $t>0.$ 
   
  In what follows, fix $p\in M$, denote $\overline{X}=X+U^*\in T_pM$ and consider $\{v_1,\ldots,v_k\}$, a $Q$-orthonormal basis of eigenvectors of $P$ with eigenvalues $\lambda_1\leq\ldots\leq\lambda_k$. Additionally, consider
  a $g$-orthonormal basis $\{e_1,....,e_n\}$ for $T_pM$, where $\{e_{k+1},...,e_n\}$ is a basis for $\cal H_p$ and $e_i = \lambda_i^{-1/2}v^{\ast}_i$ for $i\leq k$.

\begin{definition}
  We define the \textit{horizontal Ricci curvature} of $g$ at $p$ as
  \begin{equation}\label{eq:horricci}
    \Ricci^{\mathcal{H}}(\overline{X}) := \sum_{i=k+1}^nR_g(\overline{X},e_i,e_i,\overline X).
  \end{equation}
\end{definition}
   
  \begin{lemma}\label{lem:baseapropriada}
    For any $\overline{X}=X+U^*\in T_pM$ one has
    \[ \lim_{t\to \infty}\Ricci_{g_t}(\overline{X})=\Ricci_g^{\cal H}(X) + \lim_{t\to \infty}\sum_{i=1}^n z_t(C_t\overline X,C_t^{1/2}e_i) + \frac{1}{4}\sum_j\|[v_j,U]\|^2_Q.\]
    Moreover, the scalar curvature of $g_t$ is given by:
    \begin{multline}
      \label{eq:scalarcurvature}
      \mathrm{scal}_{g_t}(p) = \sum_{i,j=1}^n\kappa_0(C_t^{1/2}e_i,C^{1/2}_te_j) + z_t(C_t^{1/2}e_i,C_t^{1/2}e_j) \\+\sum_{i,j = 1}^k\frac{\lambda_i\lambda_jt^3}{(1+t\lambda_i)(1+t\lambda_j)}\frac{1}{4}\|[v_i,v_j]\|_Q^2.
    \end{multline}
  \end{lemma}
  \begin{proof}
    Using Proposition \ref{propauxiliar}, it is easy to check that $\{C_t^{-1/2}e_i\}_{i=1}^n$ is a $g_t$-orthonormal basis for $T_pM$. Moreover, $C_t^{-1/2}e_i = (1+t\lambda_i)^{1/2}e_i$ for $i \le k$ and $C_t^{-1/2}e_i = e_i$ for $i > k.$ 
    We claim that the Ricci curvature of $g_t$ satisfies:
    \begin{align}\label{eq:riccicurvature}
      \Ricci_{g_t}(\overline{X}) =& \Ricci_g^{\mathcal{H}}(C_t\overline{X}) + \sum_{i=1}^nz_t(C_t^{1/2}e_i,C_t\overline{X}) \\&+ \sum_{i=1}^k\frac{1}{1+t\lambda_i}\Big(\kappa_0(e_i,C_t\overline{X}) + \frac{\lambda_it}{4}\|[v_i,tP(1+tP)^{-1}U]\|_Q^2\Big).  \nonumber  
    \end{align}
    Indeed, note that equation \eqref{eq:curvaturaseccional} implies that
    \begin{multline*}
      \Ricci_{g_t}(C_t^{-1}\overline{X}) = \sum_{i=1}^{n}R_{g_t}(C_t^{-1/2}e_i,C_t^{-1}\overline{X},C_t^{-1}\overline{X},C_t^{-1/2}e_i) =\sum_{i=1}^n\kappa_t(C_t^{1/2}e_i,\overline{X})\\
      = \sum_{i=1}^n\kappa_0(C_t^{1/2}e_i,\overline{X}) + \sum_{i=1}^nz_t(C_t^{1/2}e_i,\overline{X}) + \frac{t^3}{4}\sum_{i=1}^k\|[PC_t^{1/2}\lambda_i^{-1/2}v_i,PU]\|_Q^2
      \\= \Ricci_g^{\mathcal{H}}(\overline{X}) + \sum_{i=1}^nz_t(C_t^{1/2}e_i,\overline{X}) + \sum_{i=1}^k\frac{1}{1+t\lambda_i}\Big(\kappa_0(e_i,\overline{X}) + \frac{\lambda_it}{4}\|[v_i,tPU]\|_Q^2 \Big).
    \end{multline*}
    Equation \eqref{eq:riccicurvature} now follows by replacing $\overline{X}$ by $C_t\overline{X}$ above.
    Besides, $\lim_{t\to\infty} C_t\overline{X} = X$. Therefore $\Ricci_g^{\mathcal{H}}(C_t\overline{X})\to \Ricci_g^\cal H(X)$ and 
    \[\sum_{i=1}^k\frac{1}{1+t\lambda_i}\kappa_0(e_i,C_t\overline{X})\to 0.\]
    Moreover,
    \[\sum_{i=1}^k\frac{t\lambda_i}{1+t\lambda_i}\frac{1}{4}\|[v_i,tP(1+tP)^{-1}U]\|_Q^2\to \sum_{i=1}^k\frac{1}{4}\|[v_i,U]\|_Q^2.\]
    Equation \eqref{eq:scalarcurvature} follows from an analogous calculation.
  \end{proof}

  \section{The behavior of Cheeger deformations at singular orbits}
  \label{sec:estimates}
   
  In this section, we shall explore the limit behavior of Cheeger deformations at a \textit{ singular orbit} (Definition \ref{def:singular}) and its influence on the sectional, Ricci, and scalar curvatures. Although relatively elementary, these consist of the main steps in the article.
   
  Let $(M,g)$ be a compact connected Riemannian manifold equipped with an isometric action by a compact Lie group $G$. We recall that, as a consequence of the Slice Theorem (see, for example \cite[Theorem 3.49, p. 65]{alexandrino2015lie}), there is an open dense convex set, $M^{reg} \subseteq M$, called \textit{the regular stratum} of the $G$-action, where the orbits have maximal dimension. In particular, for all $p,q \in M^{reg}$ it holds that $\lie m_p$ and $\lie m_q$ are isomorphic. Moreover, the restriction of the quotient projection $M^{reg} \to M^{reg}/G$ defines a Riemannian submersion (see \cite[Theorem 3.82, p. 75]{alexandrino2015lie}).
  \begin{definition}\label{def:singular}
       The orbit through any point $p \in M^{reg}$ is called either a \emph{regular orbit} or a \emph{principal orbit}.
       Both the strata $ M\setminus M^{reg}$ and any orbit through it are called \emph{singular}. 
    \end{definition}
  To motivate our interest in singular orbits, we recall that \cite[Theorem 6.3, p. 33]{SperancaCavenaghiPublished} and \cite[Proposition 6.7]{searle2015lift} implies that, after Cheeger deformation, one can uniformly make the Ricci curvature positive in any compact $K\subseteq M^{reg}$, as long as principal orbits have finite fundamental group and the quotient $GK/G$ has positive Ricci. It is then only left to produce positive Ricci around singular orbits. To this aim, one usually needs two ingredients: estimates on the Ricci curvature at singular orbits and understanding the right hypothesis needed for these bounds to guarantee positive curvature. We deal with the first ingredient in this section. For the second part, we must consider how Searle--Wilhelm's hypothesis $(SW3)$ affects the geometry. This is relatively more delicate and is done in Section \ref{sec:proofmainintro}.

  This claim about positive Ricci in $K\subseteq M^{reg}$ easily follows from the limit expression in Lemma \ref{lem:baseapropriada}. Indeed, this is achieved since $z_t$ is non-negative and the hypotheses of $\Ricci_{M/G}\geq 1$ and $|\pi_1(G/G_p)|<\infty$ are translated as positive $\Ricci^{\cal H}_g$ and positive $\sum_j\|[v_j,U]\|^2$, respectively (see sections \ref{sec:itema} and \ref{sec:proofmainintro} below). However, once in a singular orbit, the $\Ricci^{\cal H}_g$-term becomes less related to the $(SW3)$-hypothesis and we are compelled to rely on $z_t$ to produce new curvature estimates. As a result, this section is devoted to the (interesting, but usually less explored) $z_t$-tensor.
   
  Let $q\in M\setminus M^{reg}$ and consider a horizontal geodesic $s\mapsto \gamma(s)$ starting at $q$ with initial velocity $X \in \cal H_q.$ Assume further that $\gamma((0,\epsilon])\subset M^{reg}$ for some $\epsilon>0$. We will see that $R_{g_t}(X,-,-,X)$ has a very special behavior with respect to vectors which are the limit of $G_q$-action fields. As a first step, denote the restriction of the isotropy representation of $G_q$ to $\cal H_q$ by $\rho:G_q\to O(\cal H_q)$ and observe that the differential of $\rho$ at the identity $e\in G_q$ defines a linear map $d\rho:\lie g_q\to \lie o(\cal H_q)$. Since the geodesic exponential is a $G_q$-equivariant map from $T_pM$ to $M$, it follows that 
  \[d\rho(U)X=(\nabla_XU^*)_q.\] 
  %In particular, since $U^*$ is a Killing vector field, we conclude that $ d\rho(U)(X)$ is skew-symmetric.
   
  We can naturally define an extension of this skew-symmetric transformation (with fixed $X$) as the linear map 
  \begin{align*}
    \tilde S_X:\lie g&\to T_qM\\ U&\mapsto \nabla_XU^*_q.
  \end{align*} 
  We also remark that the restriction $U^*(s) := U^*(\gamma(s))$ is a Jacobi field, since isometric action vector fields are Killing.
   
  \begin{lemma}\label{lem:S_X}
    Let $q$ be a point in a singular orbit and $X\in \cal H_q$. Assume that the horizontal geodesic $\gamma : [0,\epsilon] \to M$ defined as $\gamma(s)=\exp_q(sX)$ intersects the regular stratum for any $s > 0.$ Then 
    \begin{enumerate}[$1.$]
      \item The image $\tilde S_X({\lie g_q})$ is contained in $\cal H_q$. Moreover, for $\epsilon>0$ sufficiently small, the following defines a smooth bundle on $\gamma([0,\epsilon))$
      \[\tilde{\cal H}_s=\begin{cases}\cal H_{\gamma(s)} ~~\text{if}~ s > 0\\ (\tilde S_X(\lie g_q))^\bot~ ~\text{if}~ s=0.
      \end{cases}
      \]
      \item the kernel of the restriction $\tilde S_X|_{\lie g_q}$ coincides with $\lie g_X$, the Lie algebra of $G_X=\{g\in G_q~|~\rho(g)X=X \}$.
    \end{enumerate} 
  \end{lemma}
  \begin{proof}
    We follow \cite{wilkilng-dual} (see also \cite{gw}). Consider the following family of Jacobi fields:
    \[\cal J=\{U^*|_\gamma~|~U\in\lie g\}+ \{J~|~J(0)=0,~J'(0)\in\cal H_{\gamma(0)}\}.\]
    Given $J_1,J_2\in \cal J$, as in \cite{gw}, we have
    \begin{equation}\label{eq:sympleticform}
        g(J_1'(s),J_2(s))=g(J_1(s),J_2'(s))
    \end{equation}
    for all $s$. That is, $\cal J$ is a $(n-1)$-dimensional family of normal Jacobi fields with self-adjoint Riccati operator. See \cite[Section 1.1]{guijwilhelm} for other applications coming from equation \eqref{eq:sympleticform}.
    
    From \cite{wilkilng-dual,gw} it follows that
    \begin{equation}
      \dot \gamma(s)^\bot=\mathrm{span}\{J(s)~|~J\in\cal J\}\oplus \mathrm{span}\{J'(s)~|~J\in\cal J,~J(s)=0 \}
    \end{equation}
    is an orthogonal splitting of $\dot{\gamma}(s)^\bot\subset T_{\gamma(s)}M$. In particular, if $U^*(0)=0$, then for every $V\in \lie g$ we have
    \[ g(\nabla_XU^*(0),V^*(0))=g(U^*(0),\nabla_XV^*(0))=0.\]
    Therefore, $\nabla_XU^*(0)\in \cal H_p$. 
    Observe further that $\cal J_1=\{U^*|_\gamma~|~U\in\lie g\}$ is itself a family with self-adjoint Riccati operators. Therefore, as stated in \cite{wilkilng-dual},
    \[ V(s)=\{J(s)~|~J\in\cal J_1\}+\{J'(s)~|~J\in\cal J_1,~J(s)=0\}\]
    is smooth along $\gamma$. In particular, $V(s)^\perp=\tilde{\cal H}_s$ is smooth. Item $2$ follows directly from the identity $d\rho(U)X=\nabla_XU^*(0)$. 
  \end{proof}
   
   \iffalse
  \begin{figure}[htbp]
    \centerline{\includegraphics[scale=.20]{fake.jpg}}
    \caption{In the picture, $q$ represents a point on the singular orbit colored in \textcolor{yellow}{yellow}. The \textcolor{blue}{blue} lines represent the fibers of the bundle $\tilde{\cal H}_s$, while the \textcolor{red}{red} straight line denotes the space of vectors at $q$ which come from the vertical space along $\gamma(s)$ for $s > 0$.}
    \label{fig}
  \end{figure}
   \fi
  Denote by $\lie{p}_X := \lie g_q\cap (\lie g_X)^\bot$ the $Q$-orthogonal complement of $\lie g_X$ on $\lie g_q$. Then Lemma \ref{lem:S_X} says that the restriction $\tilde S_X|_{\lie p_X} : \lie p_X \to \mathcal{H}_q$ is injective. We define:
  \begin{definition}\label{def:p_X}
    Let $q\in M$ be singular and $X \in \mathcal{H}_q$ be such that the horizontal geodesic $\exp_q(sX)$ lies in the regular part for every $s > 0.$ 
    \begin{itemize}
      \item Elements on the image $\tilde S_X(\lie p_X)$ are called \emph{fake horizontal vectors} with respect to $X$.
      \item Given $Y \in \mathcal{H}_q$, we denote by $Y_{\lie p_X}$ the unique element in $\lie p_X =  \lie g_q\cap (\lie g_X)^\bot$ such that $\tilde S_XY_{\lie p_X}$ is the orthogonal projection of $Y$ onto $\tilde S_X(\lie g_q)$. 
    \end{itemize}
  \end{definition}
   
  \begin{remark}
    The idea of taking the limit of horizontal vectors along horizontal geodesics is also present in Searle--Wilhelm \cite[section 4]{searle2015lift}. Here we provide straightforward estimates for $z_t$. 
  \end{remark}

  \subsection{The term $z_t$}
  \label{sec:zt}
  Let $X\in \cal H_q$ be a horizontal vector at $q\in M\setminus M^{reg}$ and suppose that $\exp_q(sX)$ lies in $M^{reg}$ for $s\in (0,\epsilon].$ Note that, since the regular stratum is dense, there is an open and dense set of horizontal directions with such a property and for which the following estimate holds. Therefore, it must hold for every $X,Y$.  The main result of this section is stated as:
  \begin{proposition}\label{prop:z_tgoes}
    Let $X,Y \in \cal H_q\setminus\{0\}$, where $Y$ is a fake horizontal with respect to $X$. Then,
    \begin{equation}
      z_t(X,Y) \geq 3t\dfrac{\|\tilde S_XY_{\lie p_X}\|_g^4}{\|Y_{\lie p_X}\|^2_Q}.\label{eq:blow-up}
    \end{equation}
  \end{proposition}
   
  In what follows, we prove Proposition \ref{prop:z_tgoes}. To this aim, for any $Z\in\lie g$ define the auxiliary $1$-form
  \begin{align}
    w_Z : TM &\to \mathbb{R}\\
    \overline X &\mapsto \textstyle\frac{1}{2}g(\overline X,Z^*), \label{eq:auxiliary}
  \end{align}
  where $Z^*$ is the action vector associated with $Z$. We shall use Lemma \ref{lem:S_X} and Lemma \ref{lem:zt} below, that provides a characterization of $z_t$. For a proof, see \cite[p. 23, Lemma 3.9]{Muter}. We remark, however, that contrary to \cite{Muter,mutterz}, we use the convention $d\omega(X,Y) = X\omega(Y) - Y\omega(X) - \omega([X,Y])$.
   
  \begin{lemma}\label{lem:zt}
    For every $\overline{X}=X+U^*,~\overline{Y}=Y+V^*$, $z_t$ satisfies
    \begin{equation}\label{eq:z_t-equation}
      z_t(\overline{X},\overline{Y}) = 3t\max_{\substack{Z \in \mathfrak{g}, \\ \|Z\|_Q = 1}}\dfrac{\{dw_Z(\overline{X},\overline{Y}) + \frac{t}{2}Q([PU,PV],Z)\}^2}{tg(Z^{\ast},Z^{\ast}) + 1}.
    \end{equation} 
     
    Moreover, at regular points,
    \begin{align}
      \label{eq:dw1}dw_Z(X,V^{\ast}) &= \frac{1}{2}Xg(V^*,Z^*)=-g(S_XV^*,Z^*),\\
      \label{eq:dw2}dw_Z(X,Y)&=-\frac{1}{2}g([X,Y]^{\mathcal{V}},Z^*) = -g(A_XY,Z^*),
    \end{align}
    where $A_XY=p_{\cal V}(\nabla_XY)$, $S_XV^*=-p_\cal V(\nabla_XV^*)$ and $p_\cal V$ denotes the orthogonal projection onto $\cal V=\cal H^\bot$.
  \end{lemma}

   \begin{proof}[Proof of Proposition \ref{prop:z_tgoes}]
  Given a fake horizontal $Y \in \tilde S_X(\lie p_X)$, we define
  \begin{equation}
    Y(s) := \frac{1}{s}Y^{\ast}_{\lie p_X}(\gamma(s)),
  \end{equation}
  where $\gamma(s)$ is the geodesic generated by $X$. It is easy to see that $Y(s)$ is well-defined and, according to Lemma \ref{lem:S_X},
  \begin{equation}\label{eq:definitions}\lim_{s\to 0}Y(s) = \nabla_XY^{\ast}_{\lie p_X}(0) = d\rho(Y_{\lie p_X})X = Y.\end{equation}
   
  Our next goal consists in computing $dw_Z(Y,X)$. In what follows, we write $X(s)=\frac{d}{ds}\exp(sX)$ and, for any $W\in \lie g$, we let $W^*(s):=W^*(\gamma(s))$. 
    We claim that 
    \begin{equation}\label{eq:lim1}
      \underset{s\to 0^{+}}{\lim}dw_Z(Y(s),X(s))= g\left(Y,\nabla_XZ^*(0)\right).
    \end{equation}

    Indeed, on the one hand, for any $s>0$, equation \eqref{eq:dw1} implies that
    \begin{equation*}
  dw_Z(Y(s),X(s))=\frac{1}{2s}d\omega_Z(Y^*_{\lie p_X}(s),Z^*(s))=\frac{1}{2s}Xg(Y_{\lie p_X}^*(s),Z^*(s)).
    \end{equation*}
    On the other hand, since $Y^{\ast}_{\lie p_X}(0) = 0,$ equation \eqref{eq:dw1} gives
    \begin{equation}\label{eq:deuzero}0=dw_Z(Y_{\lie p_X}^*(0),X) = \frac{1}{2}Xg(Y_{\lie p_X}^*,Z^*)|_{s=0}.\end{equation}
Therefore, 
\begin{multline*}
      \lim_{s\to 0^{+}}dw_Z(Y(s),X(s))=\frac{1}{2}\frac{\partial^2}{\partial s^2}\Big|_{s=0}g(Y_{\lie p_X}^*(s),Z^*(s)) \\= \frac{1}{2}\left\{g\left(\frac{D^2}{ds^2}Y_{\lie p_X}^*(0),Z^*(0)\right) + g\left(Y_{\lie p_X}^*(0),\frac{D^2}{ds^2}Z^*(0)\right) + 
      2g\left(\frac{D}{ds}Y_{\lie p_X}^*(0),\frac{D}{ds}Z^*(0)\right)\right\}. 
    \end{multline*}
    The claim follows since $\frac{D^2}{ds^2}Y_{\lie p_X}^*(0) = - R(Y_{\lie p_X}^*(0),X(0))X(0)=0$ and, according to equation \eqref{eq:definitions}, one has $\frac{D}{ds}Y_{\lie p_X}^*(0) = Y$.

    Now observe that, since $z_t(X,Y)\geq 0$, Proposition \ref{prop:z_tgoes} follows immediately if $Y\bot \tilde S_X(\lie g_p)$, once $Y_{\lie p_X} = 0$. If $Y_{\lie p_X}\neq 0$, take $Z = Y_{\lie p_X}/\|Y_{\lie p_X}\|_Q$ on equation \eqref{eq:z_t-equation} and apply equation \eqref{eq:lim1}. Since $Z^*(0) = 0$ and $\frac{D}{ds}Z^{\ast}(0) = \frac{\tilde S_XY_\lie p}{\|Y_\lie p\|_Q}$, one has
    \begin{align*}
      z_t(X,Y) &\geq 3t\dfrac{g\!\left(Y,\frac{D}{ds}Z^{\ast}(0)\right)^2}{tg\!\left(Z^{\ast}(0), Z^{\ast}(0)\right) + 1}
      = 3tg\!\left(Y,\frac{\tilde S_XY_{\lie p_X}}{\|Y_{\lie p_X}\|_Q}\right)^2
      = 3t\frac{\|\tilde S_XY_{\lie p_X}\|_g^4}{\|Y_{\lie p_X}\|_Q^2}.\qedhere
    \end{align*}
  \end{proof}
   
  \section{Sufficient conditions for positive curvatures and Applications}
  \label{sec:itema}
   
  In this section, we explore the blow-up behavior of $z_t$. In particular, we prove Theorem \ref{ithm:obstruction}, item 1, and Corollary \ref{icor:quotient}. We also analyze the limit of the sectional curvature under Cheeger deformations, obtaining the condition in Theorem \ref{ithm:CDR}. Finally, we apply the limiting analysis to both cohomogeneity-one manifolds and scalar curvature, proving Theorems \ref{ithm:GZ} and \ref{ithm:LW}.
   
  \subsection{Lifting positive Ricci curvature via Cheeger deformations}
   
  This paragraph is dedicated to the proof of the following result, which concerns our first goal of this subsection.
   
  \begin{theorem}\label{thm:sufficient}
    Let $(M,g)$ be a compact Riemannian manifold with a $G$-action satisfying $(SW1)$-$(SW3)$.
    If, for every $t > 0$, there exists a unit vector $\overline X\in T_pM$ such that $\Ricci_{g_t}(\overline X) < 0,$ then there exists a point $q\in M\setminus M^{reg}$ and a non-zero vector $X\in \cal H_q$ fixed by the isotropy representation at $q$.
  \end{theorem}
   
  For the proof, we need the following auxiliary result. Although it follows from \cite[Proposition 3.3]{searle2015lift}, here we give a different proof based on equation \eqref{eq:riccicurvature}. 
  \begin{lemma}\label{lem:regularconvergence}
    Consider the Riemannian submersion $\pi{:}~(M^{reg},g)\to (M^{reg}/G,\bar g)$. If $p\in M^{reg}$ then
    \begin{equation}\label{eq:SW}\lim_{t\to \infty}\Ricci_{g_t}({X})=\Ricci_{\bar g}(d\pi X),\quad\forall X\in\mathcal{H}_p.\end{equation}
  \end{lemma}
  \begin{proof}
    Observe that
    \begin{equation*}\label{eq:lim_ric_hor}
      \lim_{t\to \infty}\Ricci_{g_t}({X}) = \Ricci_g^{\mathcal{H}}({X}) + \sum_{i=1}^n\lim_{t\to \infty}z(C_t^{1/2}e_i,{X})
    \end{equation*}
    as long as $\lim_{t\to\infty} z_t(C_t^{1/2}e_i,X)$ exists for all $i$. 
    Now, on the one hand, Lemma \ref{lem:zt} (see also \cite{Muter} or \cite{mutterz}) gives:
    \begin{align}
      \label{eq:essaajuda}z_t(X,Y)&={3}\max_{\substack{Z \in \mathfrak{g} \\ \|Z\|_Q = 1}}\left\{\dfrac{g(A_XY,Z^*)^2}{g(Z^{\ast},Z^{\ast}) + t^{-1}}\right\}
      \\z_t(X,W^*)&=3\max_{\substack{Z \in \mathfrak{g} \\ \|Z\|_Q = 1}}\left\{\dfrac{g(S_XW^*,Z^*)^2}{g(Z^{\ast},Z^{\ast}) + t^{-1}}\right\} 
    \end{align} 
    for all $Y\in\cal H_p$, $W\in\lie g$. In particular, $z_t(C_t^{1/2}e_i,X)\to 0$ for $i\leq k$, since $C_t^{1/2}e_i\to 0$, and $z_t(X,Y)\to 3\|A_XY\|_g^2$. On the other hand, O'Neill's submersion formula and equation \eqref{eq:curvaturaseccional} give
    \[R_{\bar g}(d\pi X,d\pi Y,d\pi Y, d\pi X)-R_{g_t}(X,Y,Y,X)=3\|A_XY\|^2_g-z_t(X,Y), \]
    completing the proof.
  \end{proof}
   
  \begin{proof}[Proof of Theorem \ref{thm:sufficient}]
    Assume that, for every $m\in\bb N$, there is a $g$-unit vector $\overline X_m=X_m+U^*_m\in T_{p_m}M$ such that $\Ricci_{g_m}(\overline X_m)<0$. By compactness, we can pass to a convergent subsequence to obtain $m^*>0$ and a limiting $g$-unit vector $\lim_{m\to \infty}\overline X_m=\overline X\in T_pM$, such that $\Ricci_{g_{m}}(\overline X)\leq0$ for all $m>m^*$.
    We will show that $q=\lim_{m\to \infty} p_m$ lies in a singular orbit and that $\rho(G^0_q)X=X$ where $G^0_q$ stands for the connected component of the identity of $G_q$.
     
    Indeed, using the same bases as in Lemma \ref{lem:baseapropriada}, it gives
    \begin{equation}\label{eq:contradiction}
      0 \geq \Ricci_g^{\cal H}(X) + \lim_{m \to \infty}\sum_{i=1}^n z_m(C_m\overline X,C_m^{1/2}e_i) + \frac{1}{4}\sum_j\|[v_j,U]\|^2_Q.
    \end{equation}
    Writing $\overline X = X + U^*$ and noting that, by equation \eqref{eq:z_t-equation} $z_m(C_mU,C_m^{1/2}) \to 0$ as $m\to \infty$, Lemma \ref{lem:regularconvergence} then implies that $q$ lies in the singular strata, otherwise we would have
    \[0\leq\Ricci_{\overline g}(d\pi X) + \frac{1}{4}\sum_j\|[v_j,U]\|^2_Q,\]
    which contradicts the hypothesis of positive Ricci curvature on $M^{reg}/G$. Moreover, $X\neq 0$, since, otherwise \eqref{eq:contradiction} would imply that
    \begin{equation}\label{eq:contradicts} 0 \geq \lim_{m \to \infty}\sum_{i=1}^n z_m(C_mU^*,C_m^{1/2}e_i) + \frac{1}{4}\sum_j\|[v_j,U]\|^2_Q.\end{equation}
    Equation \eqref{eq:contradicts} is a contradiction since $z_t$ is nonnegative and $\sum_j\|[v_j,U]\|^2_Q>0$ whenever $|\pi_1(Gq)|<\infty$. To conclude the latter, observe that the term $\sum_j\|[v_j,U]\|^2_Q$ is closely related to the Ricci curvature of $G/G_q$ in its normal homogeneous metric. Indeed,
    \[\Ricci_{G/G_q}(U)=\sum_j\|[v_j,U]\|^2_Q+\sum_j3\|[v_j,U]^{\lie m_q}\|^2_Q.\] 
    In particular, $\Ricci_{G/G_q}(U)=0$ if $\sum_j\|[v_j,U]\|^2_Q=0$. On the other hand, if $|\pi_1(Gq)|<\infty$, $\Ricci_{G/G_q}$ is a positive definite bilinear form. Therefore, $\sum_j\|[v_j,U]\|^2_Q=0 $ implies $U=0$, a contradiction to the fact that $\overline X=U$ is $g$-unit. 
     
    To conclude that $|\pi_1(Gq)|<\infty$, one uses the hypothesis $(SW2)$:
    For every isotropy subgroup $H$ of points in the regular stratum, it holds that $\pi_1(G/H)$ is finite. Since we can assume that, up to conjugation, $H<G_q$, the long exact sequence in homotopy $G_q/H\hookrightarrow G/H\to G/G_q$ then gives:
    \[\cdots\rightarrow \pi_1(G/H)\rightarrow \pi_1(G/G_q)\to \pi_0(G_p/H)\to \{0\} \]
    Therefore, $\pi_1(G/G_q)$ is finite if, and only if, $G_q/H$ has a finite number of connected components. Which is true, since $G_q$ is a closed subgroup of a compact group. We thus conclude that $X\neq 0$. 
     
    To conclude that $X$ is fixed by $\rho(G^0_q)$, note that, otherwise, there would be a fake horizontal vector with respect to $X$ (Lemma \ref{lem:S_X}). Hence, by applying Proposition \ref{prop:z_tgoes}, we would conclude that the sum $\sum_{i=1}^n z_m(C_m\overline X,C_m^{1/2}e_i)$ diverges, contradicting inequality \eqref{eq:contradiction}.
  \end{proof}
   
  To prove Corollary \ref{icor:quotient}, we further observe that the existence of an $X$ fixed by $\rho(G_q^0)$ is related to the existence of a geodesic inside the singular strata. 
   
  \begin{proof}[Proof of Corollary \ref{icor:quotient}] Assume that there is $X\in\cal H_q\setminus\{0\}$ fixed by $\rho(G^0_q)$. Since the exponential map is $\rho$-equivariant, we have
  \[g\exp_{q}(sX)=\exp_{\rho(g)q}(\rho(g)sX)=\exp_{q}(sX)\]
  for every $s$ and $g\in G_q^0$. Therefore, $G_{\exp_q(sX)}\supseteq G_q^0$, so we conclude that $\exp_q(sX)$ is in $M\setminus M^{reg}$. Therefore, the singular strata is not composed of isolated points, as wanted.
  \end{proof}
   
  \subsection{A condition for positive sectional curvature}
   
  \label{sec:limit}
  To motivate our approach for the second goal, we make a small digression.
  As observed in \cite{Searle2015}, Cheeger deformations ``renormalizes'' the metric in the foliation induced by the $G$-orbits as $t\to \infty$. More precisely, $g_t$ approaches a metric with totally geodesic orbits as $t\to \infty$. For instance, assuming that the action is free, one can consider the connection metric $\tilde g_t$ defined through the original vertical and $g$-horizontal spaces by: $\tilde g_t(\cal H,\cal V)=0$; $\tilde g_t|_{\cal H}=g_t|_{\cal H}$; and $\tilde g_t(U^*,V^*)=t^{-1}Q(U,V)$.
   
  Roughly, $g_t$ approaches $\tilde g_t$ as $1/t^2$. So much so that, recalling that a Cheeger deformation does not produce negative curvature, one is tempted to conjecture the following:

  \begin{conjecture}\label{conj:biglie}
    Suppose that $G$ is a compact Lie group that acts freely and by isometries on a compact Riemannian manifold $(M,g)$. If $g$ has positive sectional curvature, then $M$ admits a metric with nonnegative sectional curvature where every $G$-orbit is totally geodesic.
  \end{conjecture}
   
  As pointed out in \cite{speranca_WNN}, this conjecture implies strong restrictions in the manifold, including that the related principal bundle $M\to M/G$ is \textit{fat}, i.e., for any non-zero horizontal vector $X$ and any horizontal extension $\widetilde X$ of it, it holds that $[\widetilde X,\cal H]^{\cal V} = \cal V$; (we redirect the reader to \cite{weinstein1980fat,ziller2000fatness} or \cite{gw} for definitions and results). This implies, for instance, Petersen--Wilhelm's dimension restriction conjecture (see \cite{gonzalez2017note}), for the case of principal bundles. 
   
  Although it provides a promising approach, Cheeger deformation does not solve conjecture \ref{conj:biglie} directly. Heuristically, the difference between the curvatures of $g_t$ and $\tilde g_t$ decreases as $1/t$. Since the (vertizontal) curvature of both decay with $1/t^2$, the approaching speed is not enough to guarantee that $\tilde g_t$ has non-negative curvature. More specifically, Theorem \ref{ithm:CDR} provides one (and the only one) curvature condition that is preserved under Cheeger deformation. 
   
  Here we explore the limit behavior of the Cheeger deformation and present a necessary and sufficient condition for a $G$-manifold to have positive sectional curvature after finite Cheeger deformation. Although seemly hard to apply, this condition amounts to the intrinsic geometry of the fiber (decoded by the orbit tensor $P$) being relatively decoupled from O'Neill's integrability tensor.
   
 Although the next lines are already presented in the Introduction, for the readers' convenience, we recall some notation before proceeding to the proof of the Theorem \ref{ithm:CDR}.
  Given a biinvariant inner product $Q$ on $\lie g$, for every $p\in M^{reg}$ we denote by $\Omega:\cal H_p\times\cal H_p\to \lie m_p$ the \textit{curvature 2-form} of the bundle $M^{reg}\to M^{reg}/G$. Specifically, given two horizontal vector fields $X,Y$, we define $\Omega(X,Y)$ as the unique element in $\lie m_p$ such that \[\Omega(X,Y)^*=-[X,Y]\] (see \cite[p. 70]{gw}.) It is equivalently characterized by:
  \[g(\Omega_X^*V^*,Y)=Q(\Omega(X,Y),V)=-2g(A^*_XP^{-1}V^*,Y).\]
  In particular, $-2A^*P^{-1}V^*=\Omega^*V$. 
  We now prove Theorem \ref{ithm:CDR}. That is, if $\sec_{G/G_p}>0$, then $g_t$ has positive sectional for any $t$ sufficiently large, if and only if
  \begin{multline*}
    (R_B(X,Y,Y,X)-k\|X\wedge Y\|^2_g)\\
      (\tfrac{1}{2}Q(\Hess P^{-1}(X)V,V) +\tfrac{1}{4}\|\Omega^*_XV\|_{g}^2-k\|X\|^2_gQ(P^{-1}V,V))\\\geq Q((\nabla_X\Omega)_XY,V)^2
    \end{multline*}
    for every $X,Y\in \cal H_p$, $V\in\lie m_p$ and $p\in M^{reg}$.
  \begin{remark}
    Here we think of $P$ as a function from $M$ to the linear endomorphisms of $\lie g$ (setting $P\lie g_p=\{0\}$). Given a basis of $\lie g$, $\Hess$ is computed as the Hessian of each entry in the matrix of $P$ on that basis. 
  \end{remark}

  \begin{proof}[Proof of Theorem \ref{ithm:CDR}]
    Since $\kappa_t$ is the unnormalized sectional curvature of $g_t$ reparametrized by a continuous parameter, we conclude that $\mathrm{sec}_{g_t}>0$ if and only if there exists $k > 0$, not depending on any particular plane, such that $\lim_{t\to\infty}\kappa_t(X,Y+V^*)\geq k\|X\wedge (Y+V^*)\|^2_g$. 
     
    Let $p\in M^{reg}$ and consider the plane $\{\overline X,\overline Y\}\in \mathrm{Gr}_2(T_pM)$ where $\overline{X},\overline{Y}$ is a $g$-orthonormal pair. As pointed out in Section \ref{sec:cdef}, $C^{-\h}_t\overline X,C^{-\h}_t\overline Y$ is $g_t$-orthonormal. We compute 
    \begin{equation*}
    R_{g_t}(C_t^{-\h}X,C_t^{-\h}(Y+V^*),C_t^{-\h}(Y+V^*),C_t^{-\h}X)=\kappa_t(C_t^\h X,C_t^\h(Y+V^*)).
    \end{equation*}  
    Observe that the right-hand side decays like $t^{-2}$. To correct this rate of decay, we consider the following family of linear isomorphisms $L_t : T_pM \to T_pM$: \begin{equation}L_t(Y+V^*)=Y+t^{\h}C_t^{\h}P^{-\h}V^*.\end{equation} 
    Note now that the term $t^{\h}$ compensates the decay of $C_t^\h$ while $P^{-\h}$ works as a $t$-independent {renormalization} which takes the leaf metric into account: $g(P^{-\h}V^*,P^{-\h}U^*)=Q(V,U)$. Furthermore, a direct calculation shows that $\lim_{t\to\infty}L_t = P^{-1}$. 
    %Since $\mathrm{span}\left\{X,Y+V^*\right\} = \mathrm{span}\left\{X,Y+(P^{-1}V)^*\right\}$, 
   
    Now since the parametrization $L_t$ still covers all planes of $TM$, it suffices to show that
    \begin{equation}
    \label{eq:conditionL}
  \lim_{t\to\infty}\kappa_t(L_tX,L_t(Y+V^*))\geq k\|X\wedge (Y+P^{-1}V^*)\|^2_g.
    \end{equation} 
  
  To this aim, note that
    \begin{multline}\label{eq:limitkappa}
    \lim_{t\to\infty}\kappa_{t}(X, Y+t^\h C_t^{\h}P^{-\h}V^*)=
  \kappa_0(X,Y)+2R_g(X,Y,P^{-1}V^*,X)\\+\kappa_0(X,P^{-1}V^*)+z_\infty(X,Y+P^{-1}V^*),
    \end{multline}
    where \eqref{eq:essaajuda} gives
    \begin{equation}
      \label{eq:zinfty}
       z_{\infty}(X,Y+V^*):=\lim_{t\to\infty} z_t(X,Y+V^*)=3\|A_XY+S_XV^*\|^2_g.
    \end{equation}

   Also, we can see equation \eqref{eq:limitkappa} as a quadratic function on $V^*$ by considering the change $V^* \mapsto \lambda V^*$, where $\lambda \in \mathbb{C}$, resulting in a polynomial $p(\lambda) = a\lambda^2 + b\lambda + c$, for some $a, b, c \in \mathbb{R}$.

    Since for any degree two polynomial to be non-negative we only need to understand the sign of the coefficients ($a$ and $c$) and its discriminant $b^2 - 4ac$, applying these criteria we conclude that there exists $k > 0$ such that \eqref{eq:conditionL} holds if, and only if, there exists $k > 0$ such that
    \begin{multline*}
    (K_{M/G}(X,Y)-k\|X\wedge Y\|^2_g)(\kappa_0(X,P^{-1}V^*)-3\|S_XP ^{-1}V^*\|_g^2-k\|X\|^2_gQ(P^{-1}V,V))\\\geq \frac{1}{4}(R_g(X,Y,P^{-1}V^*,X)+3 g\left(S_XP^{-1}V^*,A_XY\right))^2,
    \end{multline*}
    where we have used that $\kappa_0(X,Y) = K_{M/G}(X,Y) - 3\|A_XY\|^2$.
     
    Now let $\gamma(s)$ be the geodesic defined by $X$. Recall that the restriction of $V^*$ to $\gamma$ defines a holonomy Jacobi field (see \cite[Definition 1.4.3, p.17]{gw}), or, equivalently, $\nabla_X^{\cal V} V^*=-S_XV^*$. Moreover, for any $W\in \lie g$,
    \[g(W^*(\gamma(s)),P^{-1}V^*(\gamma(s)))=Q(PW,P^{-1}V)=Q(W,V) \]
    does not depend on $s$. 
    
    Gathering this informantion with the expression for the vertizontal curvature of a Riemannian submersion (see \cite[Corollary 1.5.1, p. 28]{gw})
    \begin{equation}\label{eq:vertizontal}
        K_{g}(X,V^*) = \langle (\nabla S_X)_XV^*,V^*\rangle - \|S_XV^*\|^2 + \|A^*_XV^*\|^2
    \end{equation}
    one gets
    \begin{equation}\label{eq:Hess}
  \kappa_0(X,P^{-1}V^*)-3\|S_XP ^{-1}V^*\|^2= \tfrac{1}{2}Q(\mathrm{Hess} P^{-1}(X)V,V) +\tfrac{1}{4}\|\Omega^*_XV\|_{g}^2
    \end{equation} where we have used that
    \begin{equation}\label{eq:dualderivative}
      \nabla_X^{\cal V}(P^{-1}_sV)^*=S_X(P^{-1}_sV)^*
    \end{equation}
    (we refer to \cite{speranca_oddbundles} for further details). 
    
    We remark that equation \eqref{eq:dualderivative}, in its turn, follows from
    \begin{align*}
      g(\nabla_X^{\cal V}(P^{-1}_sV)^*,W^*) &= Xg((P^{-1}_sV)^*,W^*) - g((P^{-1}_sV)^*,\nabla^{\cal V}_{X}W^*)\\
      &= -g((P^{-1}_sV)^*,\nabla^{\cal V}_XW^*)\\
      &= g((P^{-1}_sV)^*,(S_{X}W^*))\\
      &= g(S_{X}(P^{-1}_sV)^*,W^*).
    \end{align*}
     
    The term $R_g(X,Y,P^{-1}V^*,X)+3g({S_X\left(P^{-1}V\right)^*,A_XY})$ appears in \cite{speranca_WNN} where it is proved to be independent of $P$. Here we identify this term with $-2Q(\frac{d}{ds}\Omega(X,Y),V)$. We assume $\nabla_XX=\nabla_XY^{\cal H}=0$ until the end of the proof and recall that (see \cite[section 1.9]{gw}):
    \begin{equation*}  R_g(X,Y,\left(P_s^{-1}V\right)^*,X)= g((\nabla_XA)_XY,P_s^{-1}V^*)-2g(S_X\left(P_s^{-1}V\right)^*,A_XY).
    \end{equation*}
    Therefore, \eqref{eq:dualderivative} gives
    \begin{align*}
      %\begin{split}
      R_g(X,Y,(P_s^{-1}V)^*,X)+&3g({S_X(P^{-1}V)^*,A_XY}) \\&= g((\nabla_XA)_XY,(P^{-1}V)^*) + g(S_X(P^{-1}V)^*,A_XY)\\
      &= g\left(\nabla_X(A_XY),(P^{-1}V)^*\right) +  g(S_X(P^{-1}V)^*,A_XY)\\ &=Xg(A_XY,(P^{-1}V)^*)\\
      &=-2XQ(\Omega(X,Y),V)\\
      &=-2Q(V,\nabla_X\Omega(X,Y)). \qedhere
    \end{align*}
  \end{proof}
   
  \subsection{Positive Ricci curvature of cohomogeneity-one manifolds}
   
    Our next step is to prove Theorem \ref{ithm:GZ}. Cohomogeneity-one manifolds present a special challenge. Here we again consider the limit behavior of the Cheeger deformation, but now considering that there is no horizontal curvature.
We deal with the case where the quotient space is a closed interval $M/G=[0,R]$ and denote $\pi:M\to [0,R]$ as the quotient map. Let $\gamma:[0,R]\to M$ be a horizontal geodesic satisfying $\pi\circ \gamma(s)=s$. Denote $\dot \gamma=X$. From the $G$-invariance of the curvature, it suffices to show that $\Ricci_{g_t}>0$ along $\gamma$. As in \cite{Grove2002}, we use diagonal metrics and (to a lesser extent) a disk bundle argument to provide suitable smoothness conditions at the singular orbits. In the process, we also show that not every invariant metric has positive Ricci curvature, even after Cheeger deformation. This points to the fact that $M/G$ is flat, contrasting with Theorem \ref{ithm:searleintro}.
   
  The natural steps would be to produce an initial metric $g$ and then use a sequential argument in the spirit of the proof of Theorem \ref{thm:sufficient}. To motivate the construction of the initial metric, we invert the order of the steps, making explicit which conditions $g$ must satisfy before producing it. More precisely, Proposition \ref{prop:cohom1} below provides necessary and sufficient conditions to the existence of the desired metric with positive Ricci curvature. Theorem \ref{ithm:GZ} then follows verifying such conditions in a standard way. That said, the merit in our approach relies only on establishing such a proposition.
  
  In what follows, to emphasize the dependence of the orbit tensor along the points in $M$, we denote the orbit tensor at $\gamma(s)$ by $P_s$. 
  \begin{proposition}\label{prop:cohom1}
    Suppose that $(M,g)$ is a cohomogeneity one manifold whose orbits have finite fundamental group. Then $g_t$ has positive Ricci curvature for all large $t$ if and only if there is $c>0$ such that, for all $s\in (0,R)$,
    \[\frac{d^2}{ds^2}\tr P^{-1}_s\geq c.\] 
  \end{proposition}
  \begin{proof}
    We only prove the \textit{if} part. The converse can be proven in a similar way. 
     
    Consider the function $F(\overline{Y},t)=t\Ricci_{g_t}(\overline{Y})$
    where $\overline{Y}$ has unit $g$-norm (the $t$-factor is essential since the $\Ricci_{g_t}(\overline{Y})$ might decay.) Arguing by contradiction, we assume that for every $m \in \mathbb{N}$ there is $\overline{Y}_m$ such that $F(\overline{Y}_m,m)\leq 0$. By compactness, passing to a converging subsequence, we obtain a $g$-unit vector $\overline{Y}$ such that $\lim_{m\to\infty} F(\overline{Y},m)\leq 0$. 
     
    Denote by $p\in\gamma([0,R])$ the footpoint of $\overline{Y}$ and decompose $\overline{Y}=Y+V^*$ into its horizontal and vertical components. We first note that $p$ must lie in a regular orbit: recall that $G_{\gamma(0)},G_{\gamma(R)}$ act transitively on their respective horizontal spheres, as noted in \cite{gz,Grove2002}. Therefore, $d\rho(\lie g_p)Y\neq \{0\}$ whenever $p\in M\setminus M^{reg}$. Hence, Lemma \ref{lem:baseapropriada} and Proposition \ref{prop:z_tgoes} imply that $\Ricci_{g_t}(\overline{Y})$ would have an unbounded $z_t$-term whether $Y\neq 0$ (by summing it with any fake horizontal) or $V^*\neq 0$ (by summing it with $Y$). 
    Therefore, $p\in M^{reg}$, as claimed.
     
    Now, since $p$ lies in a regular orbit, $\Ricci_g^\cal H(Y)=0$ and conclude that $V^*=0$; otherwise we would have $\lim\limits_{m\to \infty} \Ricci_{g_m}(\overline{Y})\geq \frac{1}{4}\sum \|[v_i,V]\|^2_Q$, which is positive whenever $V^*\neq 0$.
  Therefore, the limit vector $Y$ is precisely the velocity $\dot \gamma = X.$ To finish the proof, observe that equation \eqref{eq:riccicurvature} gives
    \begin{multline}
      0\geq \lim_{m\to\infty}F(X,m)= \lim_{m\to \infty} \sum_{i=1}^k\frac{m}{1+m\lambda_i}\Big(\kappa_0(e_i,X)+z_m(e_i,X)\Big)  \nonumber\\
      =\sum_{i=1}^k\Big(\kappa_0({\lambda_i^{-1}}v_i^*,X)+3\|S_X({\lambda_i^{-1}}v_i^*)\|^2_g\Big)=\sum_{i=1}^k\Big(\kappa_0(P^{-1}_sv_i^*,X)+3\|S_X(P^{-1}_sv_i^*)\|^2_g \Big).
    \end{multline}
    This finishes the proof since \eqref{eq:Hess} gives
    \begin{equation*}\label{eq:dualhol}
    \kappa_0(P^{-1}_sv_i^*,X)+3\|S_X(P^{-1}_sv_i^*)\|^2_g=\frac{1}{2}\frac{d^2}{ds^2}Q(P^{-1}_sv_i,v_i) \geq c.\qedhere
    \end{equation*}
  \end{proof}
   
  We are now ready to prove Theorem \ref{ithm:GZ}.
   
  \begin{proof}[Proof of Theorem \ref{ithm:GZ}]
    Let $\pi:M\to [0,R]$ be the quotient projection and $\gamma:[0,R]\to M$ be a horizontal geodesic satisfying $\pi\circ \gamma(s)=s$. Set $\dot \gamma=X$ and recall that both $\lie m_{\gamma(s)}$ and $G_{\gamma(s)}$ do not depend on $s\in (0,R)$. Write $\lie m_{\gamma(R/2)}=\lie n$ and $G_{\gamma(R/2)}=H$. A \textit{diagonal metric} is a metric satisfying $P_s|_{\lie n_i}=f_i(s)^2\id$, where $\lie n=\oplus \lie n_i$ is an $Ad(H)$-invariant decomposition of $\lie n$. 
     
    Since (see for instance Lemma 3.3 in \cite{Grove2002})
    \[\frac{d^2}{ds^2}\tr P_s^{-1}= \sum \frac{d^2f_i^{-2}}{ds^2}= \sum -2\frac{f_i''}{f_i^3}+6\frac{(f_i')^2}{f_i^4},\]
    we conclude Theorem \ref{ithm:GZ} by showing that we can choose $\{f_i\}$ as a set of concave functions such that, for each $s\in[0,R]$, there is at least one index, say $i$, for which $f_i''<0$. 
    Such functions can be easily produced by imitating bundle metrics on the tubular neighborhoods of each singular orbit, then manipulating the behavior of the functions in the regular part.

    For instance, we choose the decomposition $\lie n=\lie n_-+\lie n_0+\lie n_R+\lie n_+$ where: $\lie n_-=\lie g_{\gamma(0)}\cap \lie g_{\gamma(R)}$ is composed of vectors whose action fields vanish at both $0$ and $R$; $\lie n_0=\lie g_{\gamma(0)}\cap (\lie n_-)^\bot$ (respectively, $\lie n_R=\lie g_{\gamma(R)}\cap (\lie n_-)^\bot$), composed of vectors whose action fields vanish only at $0$ (respectively, at $R$); $\lie n_+=\lie n \cap (\lie g_{\gamma(0)}+\lie g_{\gamma(R)})^\perp$ is composed of fields that never vanish. Instead of choosing one function for each $\lie n_i$, we choose $f_0,f_-,f_+,f_R$ relative to the new decomposition.
     
    Based on the geometry of disk bundles, see for instance the content of Lemma 3.3 in \cite{Grove2002}, it is known that there is some $A> 0$ for which we can impose that: $f_-,f_0$ agree with $A^{-1}\sin(At)$ near $0$; $f_-,f_R$ agree with $A^{-1}\sin(A(R-t))$ near $R$; and $f_i$ is constant near an extreme point where it does not vanish (see \cite{Grove2002}.) 
    Since we are assuming that both $\gamma(0),\gamma(R)$ are at the singular orbits, both $\lie n_0+\lie n_-,\lie n_R+\lie n_-$ are non-trivial, therefore, we guarantee that there are strictly convex functions around each singular orbit. 
    Fixing these conditions, it is straightforward to manipulate the functions $f_i$ in the interior of $(0,s)$ so that the desired concavity conditions hold. 
  \end{proof}
   
  \subsection{Non-abelian symmetry and positive scalar curvature}
   
  \label{sec:LawsonYau}
   
  Lawson and Yau in \cite{lawson-yau} prove that any Riemannian manifold $(M,g)$ endowed with the action of a non-Abelian compact connected Lie group $G$ has a metric of positive scalar curvature. The result was specially interesting at the time, since it guarantees both positive scalar curvature in a plethora of exotic spheres and lack of symmetry in the examples of exotic spheres without positive sectional curvature recently found by Hitchin \cite{hitchin1974harmonic}.
   
  The positively curved metric in \cite{lawson-yau} was obtained by reducing the group to $S^3$ (or $SO(3)$) and then proceeding in two parts: they considered a bundle-like metric in a compact subset of $M^{reg}$, with totally geodesic orbits, where they could freely apply the \textit{Canonical Variation} (see \cite[Example 2.1.1, p. 56]{gw}); and then constructed fixed metrics around the singular strata, considering delicate estimates independent of the canonical variation. In particular, during the procedure, the symmetry group is reduced to either $S^3$ or $SO(3)$. 
  To our fortune, scalar curvature is much more susceptible to become positive through Cheeger deformation than Ricci or sectional curvature. Indeed, the blow-up of $z_t$ and the last term in \eqref{eq:scalarcurvature} guarantee positive scalar curvature in a straightforward manner, as long as $\lie g$ is non-abelian\footnote{C. Searle and F. Wilhelm were already aware that Cheeger deformation yields this result.}.

  As a result, we provide a simpler proof of the result of Lawson and Yau.

  \begin{proof}[Proof of Theorem \ref{ithm:LW}]
  Assume by contradiction that there is a sequence of points $p_m\in M$ such that $\scal_{g_t}(p_m)\leq 0$. By compactness, we conclude that there is a point $p\in M$ such that 
  \[\lim_{t\to \infty}\scal_{g_t}(p)\leq 0.\] 
    On the one hand, $p$ cannot be regular since, in this case, there is $c>0$ such that $\lambda_i>c$ for every $i$. Therefore, 
      \[\frac{\lambda_i\lambda_jt^3}{(1+t\lambda_i)(1+t\lambda_j)}\|[v_i,v_j]\|_Q^2 \geq t\frac{\|[v_i,v_j]\|_Q^2}{(\frac{1}{c}+1)^2}.\]
    Assuming $\lie g$ non-Abelian, there is some pair $i,j$ such that $[v_i,v_j]\neq 0$. Since $\sum_{i,j=1}^n\kappa_0(C_t^{1/2}e_i,C^{1/2}_te_j)$ is bounded and $z_t$ is non-negative, by recalling equation \eqref{eq:scalarcurvature}, we conclude that one can take $t$ large enough so that $\scal_{g_t}(p)$ is positive.
     
    On the other hand, $p$ cannot be in a singular orbit: if $p\in M\setminus M^{reg}$, $\dim\lie g_p>0$ and there is $U$ such that $U^*$ is not the zero field but $U^*(p)=0$. We claim that $\tilde S_{\cal H_p}U\neq\{0\}$, thus Proposition \ref{prop:z_tgoes} and the arguments in the last paragraph shows that $\scal_{g_t}(p)$ is arbitrarily large as $t\to \infty$, a contradiction to the existence of the sequence $p_m$.
     
    To prove $\tilde S_{\cal H_p}U\neq\{0\}$, consider a $g$-unit horizontal vector $X$ such that $\gamma(s)=\exp_p(sX)$ is in $M^{reg}$ for some interval $(0,\epsilon)$. Since $U^*|_{\gamma}$ is a non-zero Jacobi field and $U^*(p)=0$, $\nabla_XU^*(p)=\tilde S_XU\neq 0$, as desired. 
  \end{proof}

  \section{The Ricci tensor on fixed axes}\label{sec:proofmainintro}

  Throughout this section we consider compact $(M,g)$ and $G$ also satisfying the hypotheses of Theorem \ref{ithm:searleintro}. Our main goal now is to prove Theorem \ref{ithm:mainitro}, besides giving some concrete algebraic conditions for the existence of metrics whose no Cheeger deformation lifts positive Ricci curvature. These examples are presented in Section \ref{sec:examples}. Also see Proposition \ref{prop:trace.to.Ricci} for a local description of the metrics to be produced.
   
  We start by showing that the only obstruction for lifting positive Ricci curvature is the term $\Ricci^{\cal H}_g$ at \textit{fixed axes} (see definition \ref{def:fixedaxis} below. The term $\Ricci^{\cal H}_g$ first appears in equation \eqref{eq:riccicurvature} in this paper). This emphasizes the need for using a conformal change on the metric to guarantee positive Ricci curvature everywhere. We soon shall see that the irreducibility of the isotropy representation for every orbit at the singular strata implies that Cheeger deformations do work on lifting positive Ricci curvature from the orbit space. Hence, in Section \ref{sec:combinatorial} we shall study some combinatorial and algebraic descriptions of the Ricci tensor on the fixed axes to provide a general picture of whether Cheeger deformations can be used as a single tool to produce positive Ricci curvature.
 
  Though such algebraic characterization proves to be very useful in providing the needed examples to complete Theorem \ref{ithm:mainitro}, it is natural to translate it in terms of some geometric data. This is the content presented in Section \ref{sec:recognizingeometric}.
  
  \begin{definition}\label{def:fixedaxis}
    Let $q\in M\setminus M^{reg}$. We say that a vector $X\in \cal H_q$ is a \emph{fixed axis} if it is fixed by $\rho(G^0_q)$, or equivalently, if $\tilde S_X=0$.
  \end{definition}
  \begin{remark}
  \begin{itemize}
      \item Since fixed axes are only defined at points in a singular orbit, from now on it is implicit we are considering only such points.
      \item The complete failure of Cheeger deformations for lifting positive Ricci curvature, as we shall see, is completely characterized by the isotropy representation action, parametrizing the so-called fake horizontal vectors. The equivalent definition of fixed axes, provided in Definition \ref{def:fixedaxis}, regarding $\tilde S_X$, provides that only the infinitesimal aspect of the isotropy group plays some role, namely, only its Lie algebra. This justifies why we restrict the definition of fixed axes only at the identity connected component of each isotropy subgroup.
  \end{itemize}
  \end{remark}
   
  \begin{proposition}\label{prop:valenada}
    Let $X \in \cal H_q$ be a fixed axis. 
    Then for every $Y \in \cal H_q$ one has that
    \begin{equation}
      z_t(X,Y) = 0, ~\forall t.
    \end{equation}
    In particular, $\lim\limits_{t\to \infty}\Ricci_{g_t}(X)= \Ricci_g^\cal H(X)$.
  \end{proposition}
  \begin{proof}
    According to equation \eqref{eq:z_t-equation} it suffices to prove that $dw_Z(X,Y) = 0$ for every $Z \in \lie g.$ To do so, we use the definition of the exterior derivative also recalling that the isometric action field $Z^*$ is a Killing vector field:
    \begin{align*}
      2dw_Z(&X,Y) = Xg(Z^*,Y) - Yg(Z^*,X) - g([X,Y],Z^*)\\
      &= g(\nabla_XZ^*,Y) + g(Z^*,\nabla_XY) - g(\nabla_YZ^*,X) - g(Z^*,\nabla_YX) - g([X,Y],Z^*)\\
      &= 2g(\nabla_XZ^*,Y) + g(Z^*,[X,Y]) -g([X,Y],Z^*)\\
      &= 2g(\tilde S_XZ,Y)\\
      &= 0.
    \end{align*}
    To conclude the limit, given $V\in\lie g$, extend $X$ and $V^*$ by the respective derivatives of $\varphi(t,s)=e^{tV}\exp_p(sX)$, so that $X$ is a horizontal field and $[X,V^*]=0$.
    We have,
    \begin{align*}
     2d\omega_Z(X,V^*)&= Xg(V^*,Z^*)-V^*g(X,Z^*)-g([X,V^*],Z^*)\\
     &=g(\nabla_XV^*,Z^*)+g(V^*,\nabla_XZ^*)\\
     &=-2g(X,\nabla_{V^*}Z^*)-g(X,[V^*,Z^*])\\
     &=-2g(X,\sigma(V^*,Z^*)),
    \end{align*} 
    where $\sigma$ is the second fundamental form of the orbit. The limit now follows from Lemma \ref{lem:baseapropriada} since both $R_g$ and $\sigma$ are bounded and $C_t|_{\cal V}\to 0$.
  \end{proof}
   
  Combining Proposition \ref{prop:valenada}, Lemma \ref{lem:regularconvergence} and Theorem \ref{thm:sufficient} we get:
   
  \begin{corollary}\label{cor:ricH}
    Let $(M,g)$ and $G$ satisfy the hypotheses $(SW1)$-$(SW2)$ of Theorem \ref{ithm:searleintro}.
    Then $(M,g_t)$ has positive Ricci curvature for every sufficiently large $t > 0$ if, and only if, $\Ricci^\cal H_g(X)>0$ for every $X\in \cal H\setminus\{0\}$.
  \end{corollary} 
   
  The proof of Theorem \ref{ithm:mainitro} follows from Corollary \ref{cor:ricH}. We shall use it to produce the mentioned examples of Riemannian manifolds satisfying the hypotheses of Searle--Wilhelm's Theorem \ref{ithm:searleintro} but which do not develop positive Ricci curvature after any Cheeger deformation. However, to proceed we need to furnish some algebraic description of the horizontal Ricci curvature on fixed axes, which justifies the forthcoming discussion.
   
  So let $X\in \cal H_q$ be a fixed axis and consider the operator $R_X := R_g(\cdot,X)X$. Given $\overline{Y}\in T_qM$, we recall that
  \begin{equation}\label{eq:forschur}
    \tilde \rho(r) R_X(\overline Y) = R_g(\tilde \rho(r)\overline{Y},\tilde \rho(r)X)\tilde \rho(r)X=R_g(\tilde \rho(r)\overline{Y},X)X=R_X(\tilde \rho(r)\overline Y), 
  \end{equation}
  for all $r\in G_q$, where $\tilde \rho:G_q\to O(T_qM)$ is the full isotropy representation (not its restriction to the horizontal space). Schur's Lemma (\cite[p. 13, Proposition 4]{scott1996linear}) then implies that $R_X$ is block diagonal on each irreducible subspace of $\tilde \rho$. We state this as a lemma:
  \begin{lemma}
    \label{lem:Schur} Let $V\subset T_qM$ be $\tilde\rho$-irreducible. Then $R_X|_{V}$ is a multiple of the identity.
  \end{lemma}
   
  Now since $\cal H_q$ is $\tilde \rho$-invariant one gets that $R_X(\cal H_q)\subseteq \cal H_q$. Moreover, once $R_X(X)=0$ it follows that $\cal H_q\cap {\mathrm{span}\{X\}}^{\perp}$ is $R_X$-invariant. From now on, unless otherwise stated, we abuse notation and denote by $R_X$ the restriction $R_X:\cal H_q\cap {\mathrm{span}\{X\}}^{\perp}\to \cal H_q\cap {\mathrm{span}\{X\}}^{\perp}$. In this notation, we benefit from the useful expression $\Ricci^\cal H_g(X)=\tr R_X$.
   
  Although Corollary \ref{cor:ricH} guarantees that $g$ fails to develop positive Ricci curvature after any finite Cheeger deformation if there exists a fixed axis $X$ such that $\Ricci^\cal H_g(X)<0$, and it is reasonably simple to produce an arbitrary metric such that $\Ricci^\cal H_g(X)<0$, condition $(SW3)$ does impose restrictions to $R_X$. 
  
  More precisely: let $l$ be the codimension of a regular orbit. Recalling that $q\in M\setminus M^{reg}$, one has that $\dim \cal H_q > l.$ Now, since $\lim_{t\to \infty} \Ricci_{g_t}(X')\geq1$ for every $X'\in\cal H$ in the regular part (see Lemma \ref{lem:regularvector}), whenever $\cal W\subseteq \cal H_q$ is a $l$-dimensional subspace which is the limit of horizontal subspaces on the regular stratum, one gets
  \begin{equation}\label{eq:conditionW}
    \Ricci_g^\cal W(X) := \sum_{i=1}^{l-1}R_g(X,e_i,e_i,X)=\lim_{t\to \infty}\sum_{i=1}^{l-1}R_{g_t}(X,e_i,e_i,X)\geq 1,
  \end{equation}
  where $\{e_0=X,e_1,...,e_{l-1}\}$ is an orthonormal basis for $\cal W$. The last equality follows since $z_t(X,e_i)=0$ for all $i$ (Proposition \ref{prop:valenada}). On the other hand, the set of such $\cal W$'s can be restricted enough so that both \eqref{eq:conditionW} and $\Ricci_g^\cal H(X)<0$ hold together.
   
 To understand the former constrained relations, let $c(s)$ be a smooth curve with $c(0)=q$ such that $c(s)\in M^{reg}$ for $s>0$. For every $s$, $\cal H_{c(s)}$ defines a curve in $\mathrm{Gr}_l(TM)$. Moreover, any limit subspace $\cal W$ of the curve $\cal H_{c(s)}\in \mathrm{Gr}_l(TM)$ as $s\to0$ must satisfy \eqref{eq:conditionW}. We conclude that every limit of a sequence $\cal H_{p_i}$, $p_i\in M^{reg}$, arises as the limit of the horizontal space along a curve and we call it as a \textit{limit horizontal space}. 
   
  Denote the set of all limit horizontal spaces at $q$ as $\widetilde{\cal W}_q$. In the next result, we give an algebraic description of such spaces.
   
  \begin{lemma}\label{lem:regularvector}
    Let $\cal W\in \widetilde{\cal W}_q$. Then there is $Y\in \cal H_q$ such that $\cal W=(\tilde S_Y\lie g_q)^{\bot}=(d\rho(\lie g_q)Y)^\bot$.
  \end{lemma}
  \begin{proof}
    Let $c(s)$ be a smooth curve, $c(0,\epsilon)\subseteq M^{reg}$, such that $\cal W$ is the limit of $\cal H_{c(s)}.$ If $\epsilon > 0$ is sufficiently small, according to Lemma \ref{lem:S_X}, a basis for $\cal H_{c(s)}$ can be taken as
    $\{\frac{1}{s}v_1^*,\ldots,\frac{1}{s}v^*_d,v^*_{d+1},\ldots,v^*_k\},$ where $\{v_1,\ldots,v_k\}$ is a basis for $\lie m_{c(\epsilon)},~v_1^*,\ldots,v_d^*\in \lie g_q$ and $v_{d+1}^*,\ldots, v^*_k$ are $Q$-orthogonal vectors to $\lie g_p.$ Since $v_i^*(c(s))\to 0$ for all $i\leq d$, it follows that 
    \[\lim_{s\to 0^+}\frac{1}{s}v_i^*(c(s)) = \nabla_{c'(0)}v_i^* = \tilde S_{c'(0)}v_i^*.\]
    The result then follows since the set $\{v_1,\ldots,v_d\}$ must span $\lie p_X.$
  \end{proof}
  
  Lemma \ref{lem:regularvector} and its proof motivates the following definition, needed to understand Theorem \ref{thm:technical}.
  \begin{definition}\label{def:regularvectors}
    A non-zero vector $Y\in \cal H_q$ which generates a subspace $\cal W$ such as in Lemma \ref{lem:regularvector} is named a \textit{regular vector} with respect to the isotropy representation $\rho$. We denote the set of these vectors by $\cal R.$
  \end{definition}
   
   Lemma \ref{lem:regularvector} implies that, since $X$ is a fixed axis (and hence $d\rho(\lie g_q)X=0$), it holds that $X\in \cal W$ for all $\cal W\in\widetilde{\cal W}_q$. Gathering all the previous discussion, to prove Theorem \ref{ithm:mainitro}, and hence, better understanding when Cheeger deformations are the only tool needed to lift positive Ricci curvature from orbit spaces, we are interested in sets of vectors fulfilling the following requirements:
  \begin{enumerate}[(a)]
    \item \label{eq:a} $\Ricci^\cal H_g(X)<0;$
    \item  \label{eq:b} $\Ricci_g^\cal W(X)\geq 1$ for all $\cal W\in\widetilde{\cal W}_q.$ 
  \end{enumerate}
  Indeed, condition \eqref{eq:a} obstructs the lift of positive Ricci while \eqref{eq:b} is necessary to fulfill condition $(SW3)$ in Theorem \ref{ithm:searleintro}.

  Now, observe that if $\cal H_q\cap {\mathrm{span}\{X\}}^{\perp}$ is $\rho$-irreducible, then $R_X$ is diagonal and hence, condition \eqref{eq:b} implies that $\Ricci^{\cal H_q}(X) > 0$. This verifies item 2. of Theorem \ref{ithm:obstruction}. For this reason, the next subsection is solely devoted to understanding the Ricci tensor evaluated in fixed axes based at points where the isotropy representation is reducible.
   
  \subsection{An algebraic description of the Ricci curvature}
  \label{sec:combinatorial}
   
  Here we present a combinatorial description of conditions \eqref{eq:a} and \eqref{eq:b}. Specifically, we search for algebraic conditions for the existence of a symmetric operator $R_X$ satisfying \eqref{eq:a} and \eqref{eq:b}. A metric realizing $R_X$ is then constructed \textit{a posteriori}. For clarity sake, we first assume that $\cal H_q=\mathrm{span}\{X\}\oplus\cal H_1\oplus\cal H_2$, where the $\cal H_i$ are $\rho$-irreducible subspaces. Lemma \ref{q:solution} provides an easily computable condition for $R_X$ with exactly two distinct eigenvalues $\lambda_1,\lambda_2$.

  \begin{remark}
    \begin{itemize}~
      \item Although the space $\cal H_q$ is metric dependent, the linear action of $G_q$ in $\cal H_q$ is equivalent to the classical isotropic representation of $G_q$ in the quotient $T_qM/T_qGq$. Therefore, all calculations can be done using a fixed complement of $T_qGq$. 
      \item As we shall see, the irreducibility of $\cal H_i$ is not needed to produce the examples in Theorem \ref{ithm:mainitro}.
    \end{itemize}
  \end{remark}

  Start by considering the complementary projections 
  \begin{equation*}\label{eq:pi}
    p_i : \cal H_q \to \cal H_i.
  \end{equation*}
  Then for any $G$-invariant metric $g$ and $Y\in\cal H_1\oplus\cal H_2$,
  \begin{equation*}
    R_X(Y) = \lambda_1\|p_1Y\|_g^2 + \lambda_2\|p_2Y\|_g^2.
  \end{equation*}
%  where $\lambda_i$ are the constants given by Schur's Lemma.
  In particular, $\Ricci^{\cal H}(X)=\lambda_1\dim \cal H_1 + \lambda_2 \dim \cal H_2$ and for each $\cal W\in\widetilde{\cal W}_q$ one has
  \begin{equation}\label{eq:riccisubspace}
    \Ricci^{\cal W}(X)_g = \sum_{i=1}^l\left(\lambda_1 \|p_1e_i\|_g^2 + \lambda_2 \|p_2e_i\|_g^2\right)=\lambda_1 \tr (p_1|_{\cal W})+\lambda_2\tr (p_2|_{\cal W}).
  \end{equation}
  \begin{remark}Although the definition of $\cal W=(d\rho(\lie g_q)Y)^\bot$ is metric dependent, the quantity
    \[\tr (p_1|_{\cal W})=\tr( p_1)-\tr (p_1|_{d\rho(\lie g_q)Y})\]
    is not, as a direct computation shows.
  \end{remark}

  The existence of $\lambda_1,\lambda_2$ satisfying \eqref{eq:a} and \eqref{eq:b} can be totally translated as properties of the set 
  \[\cal A=\{(\tr (p_1|_\cal W),\tr (p_2|_\cal W))\in \mathbb{R}^2~~~|~~~\cal W\in\widetilde{\cal W}_q\}:\]
  \begin{problem}
    \label{claim:1}
    Let $\cal A \subset \mathbb{R}^2$ be a given collection of pairs of real numbers satisfying $a + b = l-1 ~\forall (a,b) \in \cal A$. Find $\lambda_1,\lambda_2 \in \mathbb{R}$ such that
    \begin{equation}\label{eq:positivona}a\lambda_1 + b\lambda_2 \geq 1,~ \forall (a,b) \in \cal A\end{equation}
    and
    \begin{equation}\label{eq:negativona}\lambda_1\dim \cal H_1 + \lambda_2 \dim \cal H_2 < 0.\end{equation}
  \end{problem}
  Lemma \ref{q:solution} gives necessary and sufficient conditions to solve Problem \ref{claim:1}.
   
  \begin{lemma}\label{q:solution}
    Denote by $A := \dim \cal H_1$ and $B := \dim \cal H_2$. Problem \ref{claim:1} has an affirmative answer 
    if, and only if, either 
    \begin{equation}\label{eq:solution>0}
      \inf_{(a,b)\in\cal A} \{a\} > \frac{A(l-1)}{A+B},
    \end{equation} or 
    \begin{equation}\label{eq:solution<0}
      \inf_{(a,b)\in\cal A} \{b\} > \frac{B(l-1)}{A+B},
    \end{equation}
  \end{lemma}
  \begin{proof}
    From conditions \eqref{eq:positivona},\eqref{eq:negativona} it is clear that $\lambda_1\lambda_2<0$. Moreover, if $\lambda_1,\lambda_2$ gives a solution to Problem \ref{claim:1} then $-\lambda_1,-\lambda_2$ also gives a solution after interchanging the roles of $\cal H_1$ and $\cal H_2$. Since the two conditions stated in Lemma \ref{q:solution} only differ by the sign of $\lambda_1$, we can assume without loss of generality that $\lambda_1>0$ and prove \eqref{eq:solution>0}.

    To see that \eqref{eq:solution>0} is necessary, suppose that Problem \ref{claim:1} has an affirmative answer for $\lambda_1>0$. Let $(a,b)\in\cal A$. Equation \eqref{eq:positivona} gives $a > -\frac{\lambda_2}{\lambda_1}b + \frac{\epsilon}{\lambda_1}$ for some $0< \epsilon < 1$. Hence, since $a+b=l-1$, we have $a > (a-(l-1))\frac{\lambda_2}{\lambda_1} + \frac{\epsilon}{\lambda_1}$. Therefore,
    \[a > (l-1) \frac{-\frac{\lambda_2}{\lambda_1}}{( 1-\frac{\lambda_2}{\lambda_1})} + \frac{\epsilon}{\lambda_1 - \lambda_2}.\]
    On the other hand, equation \eqref{eq:negativona} gives $\frac{A}{B} < -\frac{\lambda_2}{\lambda_1}$. Since the function $f(x) = \frac{x}{x+1}$ is increasing in $]0, \infty[$ we conclude that 
    \[a > (l-1)\frac{\frac{A}{B}}{(\frac{A}{B} + 1)} + \frac{\epsilon}{\lambda_1-\lambda_2} = \frac{(l-1)A}{A+B} + \frac{\epsilon}{\lambda_1-\lambda_2},\] 
    for every $(a,b)\in\cal A$, proving condition \eqref{eq:solution>0}.
     
    Conversely, suppose that there is $\epsilon>0$ such that 
    \[a \geq \frac{A(l-1)+2\epsilon B}{A+B}\]
    for every $(a,b)\in\cal A$. Since $a+b=l-1$, we have $a(A+B)-2\epsilon B \geq A(a+b)$. Thus $\frac{a-2\epsilon}{b} \geq \frac{A}{B}$ whenever $b\neq 0$. 
    Choose $\lambda_1, \lambda_2$ such that $\lambda_1 > 0$ and $\frac{a-\epsilon}{b} \geq -\frac{\lambda_2}{\lambda_1} > \frac{A}{B}$. We obtain $0> \lambda_1A+\lambda_2B$ and $\lambda_1a +\lambda_2b\geq \epsilon$ for every $(a, b)\in\cal A$, $b\neq 0$. The result then follows a rescaling of $\lambda_1,\lambda_2$ since whenever $\lambda_1>0$ equation \eqref{eq:positivona} is automatically satisfied for $b=0$. 
  \end{proof}
   
  Having understood necessary and sufficient pointwise conditions to produce the needed counterexamples to Theorem \ref{ithm:mainitro}, we now construct a local Riemannian metric that fails to develop positive Ricci curvature after any finite Cheeger deformation:
  \begin{proposition}\label{prop:trace.to.Ricci}
    Suppose that $\cal H_q$ admits a $\rho$-invariant subspace $\cal H_1$ and let $X\notin \cal H_1$ be such that 
    \begin{equation}\label{eq:trace.to.Ricci}
      \inf_{\cal W\in\widetilde{\cal W}_q} \{\tr (p_1|_{\cal W})\} > \frac{(l-1)\dim \cal H_1}{\dim\cal H_q-1}.
    \end{equation}
    Then there is a $G$-invariant metric $g$ on a neighborhood $U$ of $q$ satisfying
    \begin{enumerate}[$1.$]
      \item $\Ricci_{U^{reg}/G}\geq 1$,
      \item $\Ricci_{g_t}(X)<0$ for every Cheeger deformation $g_t$ of $g$.
    \end{enumerate}
  \end{proposition}
  \begin{proof}
    Since \eqref{eq:trace.to.Ricci} is assumed, a direct application of Lemma \ref{q:solution} gives $\lambda_1,\lambda_2$ such that $R_X=\lambda_1 p_1+\lambda_2 p_2$ satisfies $1$ and $2$. Now, classical theory should give a $G$-invariant metric with such prescribed curvature and this would complete the proof: In Riemannian normal coordinates (using the exponential map as chart), the Taylor expansion of the metric at $0$ has identity as constant part, vanishing linear part (Christoffel's symbols vanish at $0$), and the quadratic part is the curvature. Nevertheless, it is worth pointing it out that a double warped product provides an explicit metric.
     
    Consider a biinvariant metric $Q$ on $G$ and recall that there is a $G$-invariant neighborhood of $q$ equivariantly diffeomorphic to $G\times_{\rho}\cal H_q$. It suffices to define a $\rho(G_q)$-invariant metric on $\cal H_q$ such that $R_X=\lambda_1 p_1+\lambda_2p_2$ and $\Ricci^\cal W\geq c$ for all $\cal W\in\widetilde{\cal W}_q$ for some $c>0$. 
     
    Write $\cal H_q \cong \mathbb{R}\times \mathbb{R}^{n_1}\times \mathbb{R}^{n_2}$ where $\bb R$ corresponds to the subspace spanned by the fixed axis and $\bb R^{n_1},\bb R^{n_2}$ correspond to the two $\rho(G_q)$-invariant subspaces $\cal H_1,\cal H_2$, respectively. 
    Now consider the metric:
    \begin{equation}\label{eq:bar.metric}
      \bar g = dt^2 + \phi^2\left(t\right)ds^2_{\mathbb{R}^{n_1}} + \psi^2\left(t\right)ds^2_{\mathbb{R}^{n_2}},
    \end{equation}
    where $ds^2_{\mathbb{R}^{n_i}}$ is the standard flat metric of $\mathbb{R}^{n_i}$ and  
    \begin{align}
      \label{eq:phi}    \phi(t) &=\textstyle  \frac{1}{\sqrt{\lambda_1}}\sin\left(\sqrt{\lambda_1}t\right),\\
      \label{eq:psi}    \psi(t) &= \textstyle  \frac{1}{\sqrt{-\lambda_2}}\exp\left(\sqrt{-\lambda_2}t -b\right), 
    \end{align}
    for some $b>0$ to be fixed later.
    Following the notation of \cite[p. 71]{p}, given $V,V' \in T\mathbb{R}^{n_1},~ W,W'\in T\mathbb{R}^{n_2}$ we have
    \begin{align}
      \label{eq:1}  R_{\bar g}(X,V) &= \lambda_1 X\wedge V,\\
      R_{\bar g}(X,W) &= \lambda_2 X\wedge W,\\
      \label{eq:3}  R_{\bar g}(V,V') &= \lambda_1 V\wedge V',\\
      R_{\bar g}(W,W') &= -\lambda_2\left(\exp (b-\sqrt{-\lambda_2}t)^2 -1\right) W\wedge W',\\
      \label{eq:5}  R_{\bar g}(V,W) &= -\sqrt{-\lambda_1\lambda_2}\cot(\sqrt{\lambda_1}t ) V\wedge W,
    \end{align}
    where $X\wedge Y(v)=g({Y,v})X-g({X,v})Y$. 
    In particular, $R_X=\lambda_1p_1+\lambda_2p_2$. To verify that $\Ricci_{\bar g}^\cal W\geq c$ at $q$, 
    identify $\bb R\times \bb R^{n_1}\times \bb R^{n_2}$ with $\cal H_q$ so that $q$ is the point $(t_0,0,0)$, where 
    \begin{equation}\label{eq:in1}
      -\sqrt{-\lambda_1\lambda_2}\cot(\sqrt{\lambda_1}t_0)\geq\max\left\{\lambda_1, (1-\lambda_2)\frac{\dim\cal H_q-1}{(l-1)\dim \cal H_1}\right\}.
    \end{equation}
    Choose $b>\sqrt{-\lambda_2}t_0$. Using \eqref{eq:1}-\eqref{eq:5} we conclude that
    \begin{equation*}
      \Ricci^\cal W(\alpha X+V+W)=\alpha^2\Ricci^\cal W(X)+\Ricci^\cal W(V)+\Ricci^\cal W(W),
    \end{equation*}
    for all $\alpha\in \bb R$. Using $\lambda_1,\lambda_2$ given in Lemma \ref{q:solution}, we guarantee that $\Ricci_{\bar g}^\cal W(X)\geq 1$ and $\Ricci_{\bar g}^\cal W(V)\geq \lambda_1\|V\|^2$ (this follows from equations \eqref{eq:1},\eqref{eq:3} and \eqref{eq:5}). Finally. $\Ricci_{\bar g}^\cal W(W)$ satisfies
    \begin{gather*}
      \Ricci_{\bar g}^\cal W(W)\geq K(X,W)+\tr(p_1|_{\cal W})(-\sqrt{-\lambda_1\lambda_2}\cot(\sqrt{\lambda_1}t_0))\|W\|^2\geq \lambda_1\|W\|^2,
    \end{gather*}
    where the last inequality follows from \eqref{eq:in1} since $\tr(p_1|_{\cal W})\geq \frac{(l-1)\dim \cal H_1}{\dim\cal H_q-1}$. A suitable rescaling of \eqref{eq:bar.metric} completes the proof.
  \end{proof}
   
 We now use the explicit local description given in Proposition \ref{prop:trace.to.Ricci} to prove Theorem \ref{ithm:mainitro}, by providing the following family of examples:
   
  \subsubsection{A family of counterexamples}
  \label{sec:examples}

    Here we prove Theorem \ref{ithm:mainitro}. To do so, we present a family of examples of manifolds satisfying the hypotheses of Theorem \ref{ithm:searleintro}, but that do not develop positive Ricci curvature after any Cheeger deformation. The main idea consists of combining the algebraic description given by the solution of Problem \ref{claim:1} with the local construction obtained in Proposition \ref{prop:trace.to.Ricci}. Our model example consists of a doubly warped metric on the sphere $\bb S^5\subset \bb R^6$ with the usual mono-axial $SO(3)$-action on it. 
     
    Consider $\bb S^5 \subset \mathbb{R}^6$ endowed with the standard linear $SO(3)$-action given by the inclusion $SO(3)\ni A\mapsto \mathrm{diag}(1,1,1,A)\in SO(6)$.
    Take $q = (1,0,0,0,0,0)$ and $X=(0,1,0,0,0,0)$.
    Note that $G_{q} = G_X= SO(3).$
    Moreover, the regular orbits are diffeomorphic to 2-spheres. Therefore, the dimension $l$ of the horizontal space at points in a regular orbit is $3$.
     
    Write $\cal H_q=\mathrm{span}\{X\}\oplus\cal H_1\oplus\cal H_2$ where $\cal H_1$ is the $\rho$-invariant subspace spanned by $\{(0,0,1,0,0,0)\}$ and $\cal H_2$ the $\rho$-invariant space spanned by the last three coordinates. In this way, $A=1$ and $B=3$ in Lemma \ref{q:solution}. Let us show that $\tr(p_1|_{\cal W})>\frac{1}{2}$ for all $\cal W\in\widetilde{\cal W}_q$:
     
    Given $Y\in \cal H_1\oplus \cal H_2$ one has that $d\rho(\lie{so}(3))Y\subset \cal H_2$. Therefore, $(0,0,1,0,0,0)\in (d\rho(\lie{so}(3))Y)^\bot$ for every $Y$. Using Lemma \ref{lem:regularvector} we conclude that $\tr(p_1|_\cal W)=1$ for all $\cal W$. Moreover, Lemma \ref{q:solution} guarantees that there are $\lambda_1,\lambda_2$ such that if $g$ satisfies $R_X|_{\cal H_i}=\lambda_i\id$ then $\Ricci_{g_t}(X)<0$ for all sufficiently large $t$. To construct a global explicit Riemannian metric with such a prescribed $R_X$ we consider the doubly warped metric
    \begin{equation}\label{eq:bar.metric.S5}
      g = dt^2 + \phi^2\left(t\right)ds^2_{\bb {S}^{2}} + \psi^2\left(t\right)ds^2_{\bb S^{2}}
    \end{equation}    
    where $t\in\Big(\frac{\pi}{2\sqrt{\lambda_1}},\frac{\pi}{\sqrt{\lambda_1}}\Big)$; $\phi$ is as in \eqref{eq:phi}; and $\psi(t)=\frac{1}{\sqrt{-\lambda_2}}\sinh(\sqrt{-\lambda_2}t)$ for $t\in\Big(\frac{\pi}{2\sqrt{\lambda_1}},\frac{\pi}{2\sqrt{\lambda_1}}+\epsilon\Big)$ smoothly extended so that \eqref{eq:bar.metric.S5} defines a smooth metric on $\bb S^5$. Here, the first $ds^2_{\bb S^2}$ appearing in equation \eqref{eq:bar.metric.S5} corresponds to the first three coordinates in $\bb S^5$ and the second to the last three coordinates on it. The fixed points correspond to $t=\pi/2\sqrt{\lambda_1}$, which by the continuity of the metric are such that $R_X=\lambda_1p_1+\lambda_2p_2$. Moreover, the action is polar and transitive in the second $\bb S^2$. Thus, its quotient space is a disc with the warped metric
    \[\tilde g= \textstyle dt^2+\frac{1}{\lambda_1}\sin(\sqrt{\lambda_1}t)^2ds_{\bb S^2}^2, \]
    which has positive constant curvature.
     
    The sphere $\bb S^5$ with (a rescaling of) \eqref{eq:bar.metric.S5} is our first global example of a manifold satisfying the hypotheses of Searle--Wilhelm's theorem that does not develop positive Ricci curvature after any Cheeger deformation.

    In higher dimensions, consider $\bb S^n$ with the standard mono-axial $SO(n-2)$-action fixing the first three coordinates. 
    Take $q=(1,0,0,...,0)$, $X=(0,1,0,...,0)$, $\cal H_1=\mathrm{span}\{(0,0,1,0,...,0)\}$ and $\cal H_2$ as the $\cal H_1$-orthogonal complement (in the standard metric in $\bb R^{n+1}$). Note that regular orbits have dimension $n-3$. Similarly to the previous case one proves that
    \[\tr(p_1|_{\cal W})=1>\frac{(l-1)\dim \cal H_1}{\dim \cal H_{p}-1} = \frac{2}{n-1}\]
    for every $\cal W\in\widetilde{\cal W}_q$. The analogous doubly warped metric
    has all the desired properties.

  \subsection{The algebraic characterization in geometric terms}
  \label{sec:recognizingeometric}
  
  On the one hand, we have decoupled geometric data from Ricci tensors on fixed axes to algebraic terms (Problem \ref{claim:1}) to accomplish the proof of Theorem A. On the other hand, it is natural to try to understand if such algebraic conditions could be re-translated on other geometric obstructions. With this aim, to provide a complete description of the ``failure'' or not of Cheeger deformations for providing positive Ricci curvature, we proceed proving Theorem \ref{thm:technical} below.
   
  \begin{theorem}\label{thm:technical}
    Let $(M,g)$ be a compact Riemannian manifold with an effective isometric action by a compact Lie group $G$ satisfying the hypotheses $(SW1)$-$(SW3)$ in Theorem \ref{ithm:searleintro}. Suppose that for every $t > 0$ there exists a unit vector $\overline X \in T_pM$ such that $\Ricci_{g_t}(\overline X) < 0$. Let $q\in M\setminus M^{reg}$ such that $X\in \cal H_q$ is a non-zero vector fixed by the isotropy representation at $q$, which when restricted to $\{X\}^{\perp}\cap \cal H_q$ is reducible. 
    
    If $\cal H_q\cap \mathrm{span}\{X\}^{\perp}$ has exactly $m$ $\rho$-irreducible summands, namely $\cal H_q\cap \mathrm{span}\{X\}^{\perp}=\cal H_1+\cal H_2+ \dots + \cal H_m$, then (up to changing the order of the summands), for every regular vector $Y = Y_1 + Y_2 + \dots + Y_m \in \cal H_1+\cal H_2+\dots + \cal H_m$ there exist $j_0\neq i_0 \in \{1,\dots,m\}$ such that
    \[\dim \cal H_{j_0} - \dim \rho(G_q^0)Y_{j_0} > \dfrac{(l-1)\dim \sum_{j\neq i_0}\cal H_j}{\dim \cal H_q-1},\]
    where $l$ is the codimension of a principal orbit.
  \end{theorem}

  Before proving Theorem \ref{thm:technical}, we briefly describe the corresponding analogue of Problem \ref{claim:1} in the case of an arbitrary number of $\rho$-invariant subspaces. 
   
  Let $q\in M\setminus M^{reg}$ and $X\in \cal H_q$ be a fixed axis. If there exists a decomposition of $\cal H_q\cap {\mathrm{span}\{X\}}^{\perp}$ into $m$ $\rho$-invariant subspaces $\cal H\cap {\mathrm{span}\{X\}}^{\perp}= \cal H_{1} \oplus \cal H_{2} \oplus \dots \oplus \cal H_{m}$, a similar formulation to the existence of a metric satisfying \eqref{eq:a} and \eqref{eq:b} can be given. 
  \begin{problem}
    \label{claim:12}
    Let $\cal A \subset \mathbb{R}^m$ be a given collection of $m$-tuples of real numbers satisfying $a_1 +a_2+\dots+a_m = l-1, ~\forall (a_1,\dots,a_m) \in \cal A$. Find $\lambda_1,\lambda_2,\dots, \lambda_m \in \mathbb{R}$ such that
    \begin{equation}\label{eq:positivona2}\sum_{i=1}^n a_i\lambda_i \geq 1,~ \forall (a_1,\dots,a_m) \in \cal A\end{equation}
    and
    \begin{equation}\label{eq:negativona2} \sum_{i=1}^n \lambda_i\dim \cal H_i < 0.\end{equation}
  \end{problem}

  Since the geometric data plays no role in the proof of Lemma \ref{q:solution}, we can solve Problem \ref{claim:12} in a similar way:
   
  \begin{lemma}\label{lem:natinhosolved} Let $A_1,\dots, A_m$ be positive real numbers corresponding to $\dim \cal H_1,\ldots, \dim \cal H_m$. Then Problem \ref{claim:12} has an affirmative answer if, and only if, there exist $i_0\neq j_0 \in \{1,\dots, m\}$ such that 
    \begin{equation}\label{eq:solution>02}
      \inf_{(a_1,\dots ,a_m)\in\cal A} \{ a_{j_0}\} > \frac{(l-1)\sum_{j\neq i_0} A_j}{\sum_{k=1}^m A_k},
    \end{equation}
  \end{lemma}
  \begin{proof} If Problem \ref{claim:12} has an affirmative answer with solution $\lambda_1,\dots,\lambda_n$, define $\lambda_{i_0}=\max_{i} \{\lambda_i\}$ and $\lambda_{j_0}=\min_{j} \{\lambda_{j}\}$. Inequality \eqref{eq:solution>02} then follows from Lemma \ref{q:solution} applied to the problem
    \begin{align} \Big( \sum_{i \neq j_0} a_i \Big) \lambda_{i_0} + a_{j_0}\lambda_{j_0} &\geq 1,\\
      A_{i_0}\lambda_{i_0}+\Big(\sum_{j\neq i_0} A_j \Big)\lambda_{j_0}& < 0,\end{align}
    where $(a_1,\dots,a_n) \in \tilde{\cal A}=\left\{(\sum_{i \neq j_0} a_i, a_{j_0}) : (a_1,\dots,a_m) \in \cal A\right\} \subset \mathbb{R}^2$.
     
    The converse is a straightforward consequence of Lemma \ref{q:solution} since we can choose $i_0\neq j_0$ arbitrarily and set $\lambda=\lambda_j=\lambda_i$ for all $j,i\neq j_0$.
  \end{proof}

   We finally pass to the proof of Theorem \ref{thm:technical} by recognizing $\inf\{\tr (p_1|_{\cal W})\}$ as an invariant quantity associated to the $G$-action. To this aim, fix a $\rho$-invariant inner product on $\cal H_q$ and let $\cal R$ be the regular part of $\cal H_q$ with respect to $\rho$ (recall Definition \ref{def:regularvectors}). Then:
   
  \begin{proposition}\label{propn} Suppose $d\rho(\lie g_q)X=0$ and $\cal H_q=\mathrm{span}\{X\}\oplus\cal H_1\oplus\cal H_2$ is a $\rho$-invariant decomposition. For every $Y\in \cal R$, write $Y_i$ as its $\cal H_i$-component. Then
    \begin{gather*}
      \inf_{\cal W\in\widetilde{\cal W}_q}\{\tr(p_1|_{\cal W}) \}=\inf_{Y\in\cal R}\{\dim \cal H_1 - \dim \rho(G^0_q)Y_1\}. 
    \end{gather*}
  \end{proposition}
  \begin{proof}
    Let us fix $Y=Y_1+Y_2\in \cal R$ and, for every $\alpha>0$, define $\cal W^{\alpha} {=} \left(d\rho(\lie g_q)(\alpha Y_1 + Y_2)\right)^{\perp}$. It suffices to show that
    \begin{equation}
      \inf_{\alpha>0}\{\tr(p_1|_{\cal W^\alpha}) \}=\dim \cal H_1 - \dim \rho(G_q)Y_1. 
    \end{equation}
     
    Write $\cal W^\alpha = \cal W_1\oplus \cal W_2\oplus \cal W_{12}^\alpha$ 
    where 
    \begin{align}\label{eq:w1} \cal W_1 &:= \ker ~p_2\cap \cal W^\alpha,\\
      \label{eq:w2}\cal W_2& := \ker ~p_1\cap \cal W^\alpha,\\
      \label{eq:w12}\cal W_{12}^\alpha &:= \cal W^\alpha \cap \left(\cal W_1 \oplus \cal W_2\right)^\bot,\end{align}
    where $p_i$ is the $\cal H_i$-projection. Note that $\cal W_1,\cal W_2$ do not depend on $\alpha$, since $\cal W_i=\{Z_i\in\cal H_i~|~Z_i\perp d\rho(\lie g_q) Y_i \}$. Observe also that the dimension of $\cal W^{\alpha}_{12}$ does not depend on $\alpha$. Moreover:
     
    \begin{lemma}\label{lem:inequality}
      There are constants $c,C>0$, not depending on $\alpha$, such that 
      \begin{equation*}
        \frac{1}{{1+c\alpha^2}} \leq \frac{\|p_1(v_\alpha)\|^2}{\|v_\alpha\|^2}\leq \frac{1}{{1+C\alpha^2}}
      \end{equation*}
      for every $v_\alpha\in\cal W_{12}^\alpha$.
    \end{lemma}
     
    \begin{proof} Let $\cal M=d\rho(\lie g_q)Y_1\oplus d\rho(\lie g_q)Y_2$ and note that $\cal W_{12}^\alpha\in \cal M$. Let $\cal M_{12}^\alpha$ be the orthogonal complement of $\cal W_{12}^\alpha$ in $\cal M$ and note that both $p_1|_{\cal M_{12}^\alpha},p_2|_{\cal M_{12}^\alpha}$ are isomorphisms onto their images: an element in $\ker p_1|_{\cal M_{12}^\alpha}$ lies in $\cal H_2$ and is orthogonal to both $\cal W_2$ and $\cal W_{12}^\alpha$, however $p_2(\cal W_2+\cal W_{12}^\alpha)=\cal H_2$ for every $\alpha>0$. Analogously, $\ker p_2|_{\cal M_{12}^\alpha}=\{0\}$. Moreover, given $Z_1+Z_2\in \cal M_{12}^1$, we have $\alpha Z_1+Z_2\in\cal M_{12}^\alpha$. In particular, there is an invertible linear map $T:p_1(\cal M)\to p_2(\cal M)$ such that
      \[ Z\in \cal M_{12}^\alpha \Longleftrightarrow Z= \alpha Z_1+T(Z_1),~\text{for some } Z_1\in p_1(\cal M).\]
      We conclude that:
      \begin{align*}
        W\in \cal W_{12}^\alpha &\Longleftrightarrow W=W_1-\alpha (T^{*})^{-1}W_1,~\text{for some } W_1\in p_1(\cal M).
      \end{align*} 
      Using that $\|T\|^{-1}\leq \|T^{-1}\|$ and $\|T^*\|=\|T\|$, we have
      \begin{equation*}
        \|W_1\|^2+\alpha^2\|T\|^{-2}\|W_1\|^2\leq \|W\|^2\leq \|W_1\|^2+\alpha^2\|(T^{-1})\|^2\|W_1\|^2.\qedhere
      \end{equation*}
    \end{proof}

    Now we estimate $\tr (p_1|_{\cal W^{\alpha}}).$ Since $\cal W_2\subset \ker p_1$, we take orthonormal bases $\{e_1,...,e_{d_1}\}$ and $\{e_1^{\alpha},...,e_{d}^\alpha\}$ for $\cal W_1$ and $\cal W_{12}^{\alpha}$, respectively, and consider:
    \[\tr(p_1|_{\cal W^{\alpha}}) = \sum_j^{d_{1}} \langle p_1e_j,e_j\rangle + \sum_k^{d} \langle p_1e_k^{\alpha},e_k^{\alpha}\rangle = \dim \cal W_1 + \sum_k^{d} \|p_1e_k^{\alpha}\|^2.\]
    Lemma \ref{lem:inequality} gives:
    \begin{equation}\label{eq:monotonicity}\dim \cal W_1 + \frac{1}{1+c\alpha}\dim \cal W_{12}^{\alpha}\leq \mathrm{tr}(p_1|_{\cal W^{\alpha}}) \leq \dim \cal W_1 + \frac{1}{1+C\alpha}\dim \cal W_{12}^{\alpha}.\end{equation}
    By taking $\alpha\to \infty$ we conclude that $\inf_\alpha\{\tr(p_1|_{\cal W^\alpha})\}=\dim \cal W_1$. On the other hand, $ \cal H_1=\cal W_1+d\rho(\lie g_q)Y_1$. So it follows that $ \dim \cal W_1 = \dim \cal H_1-\dim \rho(G_q)Y_1$, completing the proof.\end{proof}
  The proof of Proposition \ref{propn} can be adapted when $\cal H_q\cap {\mathrm{span}\{X\}}^{\perp}$ is decomposed into $m$ $\rho$-invariant components:
   
  \begin{proposition}\label{prop:renatinhocaracterizou} Suppose that $d\rho(\lie g_q)X=0$ and that $\cal H_q\cap {\mathrm{span}\{X\}}^{\perp}=\mathrm{span}\{X\}\oplus\cal H_1\oplus\cal H_2 \oplus \dots \oplus \cal H_m$ is a $\rho$-invariant decomposition. For every $Y\in \cal R$, write $Y_i$ as its $\cal H_i$-component. Then there exists $j_0 \in \{1,2,\dots,m\}$ for which
    \begin{gather*}
      \inf_{\cal W\in\widetilde{\cal W}_q}\{\tr(p_{j_0}|_{\cal W}) \}=\inf_{Y\in\cal R}\{\dim \cal H_{j_0} - \dim \rho(G^0_q)Y_{j_0}\}. 
    \end{gather*}
  \end{proposition}
   
  \begin{proof} Following the proof of Proposition \ref{propn}, fix $Y=Y_1+\dots+Y_m \in \cal R$. For every $\alpha>0$, define $\cal W^{\alpha} {=} \left(d\rho(\lie g_q)(\alpha Y_1 + Y_2+\dots+Y_m)\right)^{\perp}$, where $Y_i \in \cal H_{i}$.
     
    Take $j_0$ given by Lemma \ref{q:solution} and let $\cal W^\alpha = \cal W_1\oplus \cal W_2\oplus \cal W_{12}^\alpha,$
    where 
    \begin{align}\label{eq:w11} \cal W_1 &:= (\bigcap_{i \neq j_0} \ker ~p_i ) \cap \cal W^\alpha,\\
      \label{eq:w22}\cal W_2& := \ker ~p_{j_0}\cap \cal W^\alpha,\\
      \label{eq:w122}\cal W_{12}^\alpha &:= \cal W^\alpha \cap \left(\cal W_1 \oplus \cal W_2\right)^\bot,\end{align}
    where $p_i$ is the $\cal H_{i}$ projection. As in Proposition \ref{propn}, $\cal W_1,\cal W_2$ and the dimension of $\cal W_{12}^\alpha$ do not depend on $\alpha$. Moreover, Lemma \ref{lem:inequality} remains valid in this case and we can estimate $\tr(p_{j_0}|_{\cal W^{\alpha}})$ as well.
     
    The same calculation performed in Proposition \ref{propn} yields \[\inf_\alpha\{\tr(p_{j_0}|_{\cal W^\alpha})\}=\dim \cal W_1.\] Since these dimensions are finite, it is clear that $\cal W_{1} {=} \cal H_{j_0} \cap \left(d\rho(\lie g_q) Y_{j_0}\right)^\perp $, which completes the proof.
  \end{proof}
   
Now Theorem \ref{thm:technical}
    follows by combining Theorem \ref{thm:sufficient}, Propositions \ref{propn} and \ref{prop:renatinhocaracterizou} and Lemma \ref{lem:natinhosolved}.
   
  \section*{Acknowledgments}
  The authors thank the anonymous referee for criticism, comments, and suggestions, which substantially improved the exposition. They also thank M. Alexandrino, L. Gomes, M. Mazatto, and D. Fadel for helpful discussions on very early versions of this paper. The third-named author expresses gratitude to C. Searle and F. Wilhelm for bringing this problem to his attention, as well as the hospitality of Universität zu Köln. This work was supported by Fundação de Amparo à Pesquisa do Estado de São Paulo [2017/24680-1 to L.C., 2017/10892-7 and 2017/19657-0 to L.S.]; Coordenação de Aperfeiçoamento de Pessoal de Nível Superior [88882.329041/2019-01 to R.S and PROEX to L.C]; Conselho Nacional de Pesquisa [404266/2016-9 to L.S.] and the SNSF-Project 200020E\_193062 and the DFG-Priority programme SPP 2026 to L.C.
   
\bibliographystyle{alpha}
\bibliography{mainreserva}

\begin{thebibliography}{G{\'A}R18}

\bibitem[AB15]{alexandrino2015lie}
M.~M. Alexandrino and R.~G. Bettiol.
\newblock {\em Lie Groups and Geometric Aspects of Isometric Actions}.
\newblock Springer International Publishing, 2015.

\bibitem[Bes87]{besse1987einstein}
A.~L. Besse.
\newblock {\em Einstein Manifolds:}.
\newblock Classics in mathematics. Springer, 1987.

\bibitem[BW02]{wei}
I.~Belegradek and G.~Wei.
\newblock Metrics of positive {Ricci} curvature on vector bundles over
  nilmanifolds.
\newblock {\em Geometric and Functional Analysis}, 12:56--72, 05 2002.

\bibitem[CDR92]{CDR}
L.~Chaves, A.~Derdzinski, and A.~Rigas.
\newblock A condition for positivity of curvature.
\newblock {\em Boletim da Sociedade Brasileira de Matemática}, 23:153--165,
  1992.

\bibitem[Che73]{cheeger}
J.~Cheeger.
\newblock Some examples of manifolds of nonnegative curvature.
\newblock {\em J. Diff. Geom.}, 8:623--628, 1973.

\bibitem[CN15]{crowley2015new}
D.~Crowley and J.~Nordstr{\"o}m.
\newblock New invariants of {$G2$}--structures.
\newblock {\em Geometry \& Topology}, 19(5):2949--2992, 2015.

\bibitem[CS18]{SperancaCavenaghiPublished}
L.~F. Cavenaghi and L.~D. Sperança.
\newblock On the geometry of some equivariantly related manifolds.
\newblock {\em International Mathematics Research Notices}, page rny268, 2018.

\bibitem[CS22]{cavenaghi2019positive}
Leonardo~Francisco Cavenaghi and Llohann~Dallagnol Sperança.
\newblock Positive ricci curvature on fiber bundles with compact structure
  group.
\newblock {\em Advances in Geometry}, 22(1):95--104, 2022.

\bibitem[CW17]{crowley2017positive}
D.~Crowley and D.~J. Wraith.
\newblock Positive {R}icci curvature on highly connected manifolds.
\newblock {\em Journal of Differential Geometry}, 106(2):187--243, 2017.

\bibitem[Dea11]{dearricott20117}
O.~Dearricott.
\newblock A 7-manifold with positive curvature.
\newblock {\em Duke Math. J}, 158(2):307--346, 2011.

\bibitem[G{\'A}R18]{gonzalez2017note}
D.~Gonz{\'a}lez-{\'A}lvaro and M.~Radeschi.
\newblock A note on the {Petersen-Wilhelm} conjecture.
\newblock {\em Proceedings of the American Mathematical Society},
  146:4447--4458, 2018.

\bibitem[GKS20]{shankarannals}
S.~Goette, M.~Kerin, and K.~Shankar.
\newblock Highly connected 7-manifolds and non-negative sectional curvature.
\newblock {\em Annals of Mathematics}, 191(3):829--892, 2020.

\bibitem[GM74]{gromoll1974exotic}
D.~Gromoll and W.~Meyer.
\newblock An exotic sphere with nonnegative sectional curvature.
\newblock {\em Annals of Mathematics}, pages 401--406, 1974.

\bibitem[GPT98]{Gilkey1998}
P.~B. Gilkey, JeongHyeong P., and Wilderich Tuschmann.
\newblock Invariant metrics of positive {Ricci} curvature on principal bundles.
\newblock {\em Mathematische Zeitschrift}, 227(3):455--463, Mar 1998.

\bibitem[GVZ11]{grove2011exotic}
K.~Grove, L.~Verdiani, and W.~Ziller.
\newblock An {Exotic} {$T_1S^4$} with {Positive Curvature}.
\newblock {\em Geometric and Functional Analysis}, 21(3):499--524, 2011.

\bibitem[GW09]{gw}
D.~Gromoll and G.~Walshap.
\newblock {\em Metric Foliations and Curvature}.
\newblock Birkhäuser Verlag, Basel, 2009.

\bibitem[GW18]{guijwilhelm}
Luis Guijarro and Frederick Wilhelm.
\newblock Focal radius, rigidity, and lower curvature bounds.
\newblock {\em Proceedings of the London Mathematical Society},
  116(6):1519--1552, 2018.

\bibitem[GZ00]{gz}
K.~Grove and W.~Ziller.
\newblock Curvature and {Symmetry} of {Milnor Spheres}.
\newblock {\em Annals of Mathematics}, 152:331--367, 2000.

\bibitem[GZ02]{Grove2002}
K.~Grove and W.~Ziller.
\newblock Cohomogeneity one manifolds with positive {Ricci} curvature.
\newblock {\em Inventiones mathematicae}, 149(3):619--646, Sep 2002.

\bibitem[Hit74]{hitchin1974harmonic}
N.~Hitchin.
\newblock Harmonic spinors.
\newblock {\em Advances in Mathematics}, 14(1):1--55, 1974.

\bibitem[JW08]{jowr}
M.~Joachim and D.~J. Wraith.
\newblock Exotic spheres and curvature.
\newblock {\em Bulletin of the Mathematical Society}, 45:595--616, 2008.
\newblock
  \url{http://www.ams.org/journals/bull/2008-45-04/S0273-0979-08-01213-5/}.

\bibitem[LY74]{lawson-yau}
H.~B. Lawson and S.~Yau.
\newblock Scalar curvature, non-abelian group actions, and the degree of
  symmetry of exotic spheres.
\newblock {\em Comm. Math. Helv.}, 49:232--244, 1974.

\bibitem[M{\"u}t87]{Muter}
M.~M{\"u}ter.
\newblock {\em Krumm\"ungserh\"ohende deformationen mittels gruppenaktionen}.
\newblock PhD thesis, Westfälischen Wilhelms-Universität M\"unster, 1987.

\bibitem[Nas79]{nash1979positive}
J.~Nash.
\newblock Positive {R}icci curvature on fibre bundles.
\newblock {\em Journal of Differential Geometry}, 14(2):241--254, 1979.

\bibitem[Pet06]{p}
P.~Petersen.
\newblock {\em Riemannian Geometry}.
\newblock Springer, 2006.

\bibitem[Poo75]{poor1975some}
W.~A. Poor.
\newblock Some exotic spheres with positive ricci curvature.
\newblock {\em Mathematische Annalen}, 216(3):245--252, Oct 1975.

\bibitem[PW14]{pro2014riemannian}
C.~Pro and F.~Wilhelm.
\newblock Riemannian submersions need not preserve positive {Ricci} curvature.
\newblock {\em Proceedings of the American Mathematical Society},
  142(7):2529--2535, 2014.

\bibitem[Spe17]{speranca_WNN}
L.~D. Sperança.
\newblock An intrinsic curvature condition for submersions over {Riemannian
  manifolds}, 2017.
\newblock eprint arXiv:1706.09211.

\bibitem[Spe18]{speranca_oddbundles}
L.~D. Speran{\c c}a.
\newblock On riemannian foliations over positively curved manifolds.
\newblock {\em The Journal of Geometric Analysis}, 28(3):2206--2224, 2018.

\bibitem[SS96]{scott1996linear}
L.~L. Scott and J.~P. Serre.
\newblock {\em Linear Representations of Finite Groups}.
\newblock Graduate Texts in Mathematics. Springer New York, 1996.

\bibitem[SSW15]{Searle2015}
C.~Searle, P.~Sol{\'o}rzano, and F.~Wilhelm.
\newblock Regularization via {Cheeger} deformations.
\newblock {\em Annals of Global Analysis and Geometry}, 48(4):295--303, Dec
  2015.

\bibitem[SW15]{searle2015lift}
C.~Searle and F.~Wilhelm.
\newblock How to lift positive {R}icci curvature.
\newblock {\em Geometry \& Topology}, 19(3):1409--1475, 2015.

\bibitem[Wei80]{weinstein1980fat}
A.~Weinstein.
\newblock Fat bundles and symplectic manifolds.
\newblock {\em Advances in Mathematics}, 37(3):239--250, 1980.

\bibitem[Wil01]{wilhelm-lots}
F.~Wilhelm.
\newblock Exotic spheres with lots of positive curvatures.
\newblock {\em J. Geometric Anal.}, 11:161--186, 2001.

\bibitem[Wil07]{wilkilng-dual}
B.~Wilking.
\newblock A duality theorem for {Riemannian} foliations in nonnegative
  sectional curvature.
\newblock {\em Geom. Func. Anal.}, 17:1297--1320, 2007.

\bibitem[Wra97]{wraith1997}
D.~J. Wraith.
\newblock Exotic spheres with positive {R}icci curvature.
\newblock {\em J. Differential Geom.}, 45(3):638--649, 1997.

\bibitem[Wra07]{wraith2007new}
D.~J. Wraith.
\newblock New connected sums with positive {R}icci curvature.
\newblock {\em Annals of Global Analysis and Geometry}, 32(4):343--360, 2007.

\bibitem[Zil]{mutterz}
W.~Ziller.
\newblock On {M. M\"uter's Ph. D. Thesis}.
\newblock Available in http://www.math.upenn.edu/~wziller/research.html.

\bibitem[Zil00]{ziller2000fatness}
W.~Ziller.
\newblock Fatness revisited.
\newblock Available in https://www2.math.upenn.edu/~wziller/papers/Fat-09.pdf,
  2000.

\end{thebibliography}
   
\end{document}